\documentclass[12pt]{amsart}
\usepackage{geometry}
\usepackage{comment}
\usepackage{amssymb}
\usepackage{enumerate}
\usepackage{stmaryrd}
\usepackage{graphicx}
\usepackage[colorlinks=true, citecolor=blue]{hyperref}
\usepackage{authblk}
\usepackage[normalem]{ulem}
\usepackage[french, english]{babel}

\usepackage[nocompress]{cite}

\usepackage{ifdraft}
\ifdraft{
\newgeometry{asymmetric}
\usepackage[notref, notcite]{showkeys}
\usepackage[bordercolor=white, color=white, colorinlistoftodos]{todonotes}
}{
\usepackage[disable]{todonotes}
}

\makeatletter\providecommand\@dotsep{5}\def\listtodoname{List of Todos}\def\listoftodos{\hypersetup{linkcolor=black}\@starttoc{tdo}\listtodoname\hypersetup{linkcolor=blue}}\makeatother

\newtheorem{lemma}{Lemma}
\newtheorem{proposition}{Proposition}
\newtheorem{theorem}{Theorem}
\newtheorem{definition}{Definition}

\newtheorem{corollary}{Corollary}
\newtheorem{remark}{Remark}

\def\C{\mathbb C}
\def\R{\mathbb R}

\def\N{\mathbb N}
\def\to{\rightarrow}

\def\p{\partial}
\DeclareMathOperator{\supp}{supp}

\newcommand{\pair}[1]{\left\langle #1 \right\rangle}
\newcommand{\norm}[1]{\left\|#1 \right\|}

\def\expec{\mathbb{E}}
\def\D{\mathcal D}

\def\bra{\langle}
\def\cet{\rangle}

\def\Src{\mathcal X}
\def\P{\mathbb P}
\def\W{\mathbb W}

\begin{document}
\title[Imaging with an white noise source]{Correlation based passive imaging with a white noise source}
\author[T. Helin, M. Lassas, L. Oksanen, T. Saksala]{}
\date{\today}
\maketitle
\centerline{\scshape Tapio Helin$^{\rm a}$, Matti Lassas$^{\rm a}$, Lauri Oksanen$^{\rm b}$ and Teemu Saksala$^{\rm a}$}
\medskip
{\footnotesize $^{\rm a}$Department of Mathematics and Statistics, University of Helsinki, Finland \\
\indent$^{\rm b}$Department of Mathematics, University College London, UK}
\email{teemu.saksala$@$helsinki.fi}

\selectlanguage{english}
\begin{abstract}
Passive imaging refers to problems where waves generated by unknown sources are recorded and used to image the medium through which they travel. 
The sources are typically modelled as a random variable
and it is assumed that some statistical information is available.
In this paper 
we study the stochastic wave equation $\p_t^2u - \Delta_g u = \chi W$, where $W$ is a random variable with the white noise statistics on $\R^{1+n}$, $n \ge 3$, 
$\chi$ is a smooth function vanishing for negative times and outside a compact set in space, and $\Delta_g$ is the Laplace--Beltrami operator
associated to a smooth non-trapping Riemannian metric tensor $g$ on $\R^n$. The metric tensor $g$ models the medium to be imaged, and 
we assume that it coincides with the Euclidean metric outside a compact set. 
We consider the empirical correlations on an open set $\mathcal{X} \subset \R^n$,
\begin{equation*}
	C_T(t_1, x_1, t_2, x_2)
	= \frac 1 T \int_0^T u(t_1+s,x_1) u(t_2+s,x_2)ds,	
\quad t_1,t_2>0,\ x_1,x_2\in \mathcal{X},
\end{equation*}
for $T>0$.
Supposing that $\chi$ is non-zero on $\mathcal{X}$ and constant in time after $t  > 1$,
we show that in the limit $T \to \infty$, the data $C_T$ becomes statistically stable, that is, independent of the realization of $W$. 
Our main result is that, with probability one, this limit determines the Riemannian manifold $(\R^n,g)$ up to an isometry.
To our knowledge, this is the first result showing that a medium can be determined in a passive imaging setting, without assuming a separation of scales.
\end{abstract}

\section{Introduction}

In passive imaging, waves generated by unknown sources are recorded and used to image the medium through which they travel.
Passiveness refers to the observer having only little or no control over the source (think earthquakes in seismic imaging). However, some statistical information of the source may be available and it can be useful to model the source as a random variable: while the statistics of the random variable is known, its realization remains unknown.

Passive imaging has had a fundamental impact to seismic and various other imaging modalities. We refer to the recent book by Garnier and Papanicolaou \cite{garnier2016passive} for an extensive review of the field. 
The previous mathematical theory is, to a large extent, based on
assuming some physical scaling regime. Such an approach has produced a number of important and efficient numerical methods. 
However, our key finding in the present paper is that exact recovery of an unknown medium is also possible without any scaling assumptions.
The proof of this is based on a reduction to a deterministic inverse problem.

In this work we consider the wave equation
\begin{align}
\label{eq:problem}
&\p_t^2 u(t,x) - \Delta_g u(t,x) = \chi(t,x) W(t,x) \quad \text{in $\R_+^{1+n}= (0,\infty) \times \R^n$},
\\
&u|_{t = 0} = \p_t u|_{t=0} = 0,\nonumber
\end{align}
where $\Delta_g$ is the Laplace--Beltrami operator corresponding 
to a smooth time-independent Riemannian metric $g$ on $\R^n$. In coordinates $(x_j)_{j=1}^n$ this operator has the following representation.
\begin{equation*}
	\Delta_g = \sum_{j,k=1}^n |g|^{-1/2} \frac{\partial}{\partial x^j}\left(|g|^{1/2} g^{jk} \frac{\partial}{\partial x^k} u\right),
\end{equation*}
where $[g_{jk}]_{j,k=1}^n=g(x)$, $|g| = {\rm det}(g_{jk})$ and $[g^{jk}]^n_{j,k=1} = g(x)^{-1}$. We assume that our source $W$ is a realization of a Gaussian white noise random variable on $\R^{1+n}$. 
Moreover, $\chi$ stands for a smooth function
\begin{equation*}
	\chi(t,x) = \chi_0(t) \kappa(x),
\end{equation*}
such that $\chi_0 \in C^\infty(\R)$ and
\begin{equation*}
	\chi_0(t) = \begin{cases}
	0, \quad t \leq 0, \\
	1, \quad t \geq 1,
	\end{cases}
\end{equation*}
and $\kappa \in C_0^\infty(\R^n)$.
%
We assume that there exists an open and non-empty set $\Src\subset \R^n$ where $\kappa$ is non-vanishing.
The source $\chi W$ can be modelled as a random variable taking values in a local Sobolev space with negative index, and the same is true for the solution $u$. 
Contrary to papers such as \cite{millet1999stochastic,ondrejat,conus2008non},
we do not consider $t \mapsto u(t, \cdot)$ as a random process.

The problem we study is the following: suppose we can record the empirical correlation
\begin{align}
\label{covar_data}
	C_T(t_1, x_1, t_2, x_2)
	= \frac 1 T \int_0^T u(t_1+s,x_1) u(t_2+s,x_2)ds,
\end{align}
for $t_1,t_2 > 0$, $x_1, x_2 \in \Src$ and $T > 0$. What information does this data yield regarding the metric $g$? For any finite $T$, the correlation $C_T$ is random in the sense that it depends on the realization of the source. A fundamental part of our result below is to show that this data becomes \emph{statistically stable}, i.e. independent of the realization, as $T$ increases. More precisely, we show that the limit 
$$
\lim_{T\to\infty} \pair{C_T, f \otimes h}_{\D'\times C_0^\infty(\R^{2+2n})},
\quad f,g \in C_0^\infty(\R^{1+n}),
$$ 
is deterministic, see Theorem \ref{th_inner_prods_by_ergodicity} below.
Thereafter, the paper is devoted to showing that this stability enables the recovery of $g$:

\begin{theorem}
\label{th_main}
Let $n\geq 3$. Suppose that $g$ is non-trapping and that $g$ coincides with the Euclidean metric outside a compact set. 
Let $u=\mathbb U(\omega)$ be the solution of (\ref{eq:problem}) 
where $W = \mathbb W(\omega)$ is a realization of the white noise $\mathbb W$ on $\R^{1+n}$.
Then with probability one, the empirical correlations \eqref{covar_data}
defined in the sense of generalized random variables in $\D'((\R \times \Src)^2)$ for $T > 0$, 
determine the Riemannian manifold $(\R^n,g)$ up to an isometry.
\end{theorem}
Recall that a metric tensor $g$ on $\R^n$ is non-trapping if for each compact $K \subset \R^n$ there exists $T>0$ such that for each $(p,\xi)\in T\R^n, \:p\in K,  \:\|\xi\|_g=1$,  it holds that $\gamma_{p,\xi}(t)\notin K$ when $t \ge T$. Here we denote by $\gamma_{p,\xi}$ the unique maximal geodesic of metric $g$ that satisfies the following initial conditions
$$
\gamma_{p,\xi}(0)=x \hbox{ and } \dot{\gamma}_{p,\xi}(0)=\xi.
$$

Note that the covariance data (\ref{covar_data}) is determined by the measurement $u|_{(0,\infty) \times \Src}$.
This implies the following corollary:

\begin{corollary}
The measurement $u|_{(0,\infty) \times \Src}$, with a single realization of the white noise source, determines the Riemannian manifold $(\R^n,g)$, up to an isometry, with probability one under the assumptions of Theorem \ref{th_main}.
\end{corollary}

The statistical stability of $C_T$, $T > 0$, allows us to reduce the passive imaging problem to a deterministic inverse problem, that we then solve. As this deterministic problem is of independent interest, we solve it in a more general geometric setting. Moreover we do not assume that the Riemannian manifold, we are considering about, is Euclidean outside some compact set. 
\begin{theorem}
\label{th:main_geom}
Let $(N,g)$ be a smooth and complete Riemannian manifold of dimension $n\geq 2$. Let $\Src\subset N$ be an open and nonempty set. Consider the following initial value problem for the wave equation
\begin{align}
\label{eq:problem_RM}
&\p_t^2 w(t,x) - \Delta_g w(t,x) = f,\quad \textrm{in} \:\: (0,\infty) \times N,
\\
&w|_{t = 0} = \p_t w|_{t=0} = 0.\nonumber
\end{align}
Let $\Lambda_\Src:C^{\infty}_0((0,\infty)\times \Src) \rightarrow C^{\infty}((0,\infty)\times \Src)$ be the local source-to-solution operator defined by
$$
\Lambda_\Src f=w|_{(0,\infty)\times \Src}.
$$
Then the data $(\Src,\Lambda_\Src)$ determines $(N,g)$ up to an isometry. More precisely this means the following: 

Let $(N_i,g_i), \: i=1,2$, be a smooth and complete Riemannian manifold. Let $\Src_i\subset N_i$ be open and nonempty, and assume that there exists a diffeomorphism 
\begin{equation}
\label{diffeo phi}
\phi:\mathcal{X}_1 \rightarrow \Src_2
\end{equation}
that satisfies
\begin{equation}
\label{eq: phi pulls back Lambda}
\phi^\ast (\Lambda_{\Src_2} f)= \Lambda_{\Src_1} (\phi^\ast f), \quad \hbox{for all } f\in C^{\infty}_0((0,\infty)\times \Src_2).
\end{equation}
Then $(N_1,g_1)$ and $(N_2,g_2)$ are Riemannian isometric.
\end{theorem}

Above the pullback $\phi^\ast$ of $\phi$ is defined by 
$\phi^\ast f = f \circ\widetilde{\phi}$,
where $\widetilde \phi$ is the lift of $\phi$ on $(0,\infty) \times \Src_1$, that is, $\widetilde \phi(t,x)=(t,\phi(x))  \: t>0, \: x\in \Src_1$.

Lastly we will point the connection of Theorem \ref{th:main_geom} to the following Inverse spectral problem of Laplace-Beltrami operator.

\begin{corollary}
\label{co:spectral problem}
Let $(N,g)$ be a smooth and compact Riemannian manifold of dimension $n\geq 2$ with out boundary. Let $\Src\subset N$ be an open and nonempty set. Let $(\varphi_k)_{k=1}^\infty \subset C^\infty(N)$ be the collection of orthonormal eigenfunctions of operator $\Delta_g$ in $L^2(N)$. Let $(\lambda_k)_{k=1}^\infty$ be the collection of corresponding eigenvalues of $\Delta_g$. Then the Spectral data 
\begin{equation}
\label{Spectral data}
(\Src,(\varphi_k|_{\Src})_{k=1}^\infty, (\lambda_k)_{k=1}^\infty)
\end{equation}
determines $(N,g)$ up to isometry.
\end{corollary}

\subsection{Outline the paper}
We begin by showing that the empirical correlation $C_T$ is well-defined in Section \ref{sec:direct_problem}. 
In Section \ref{sec:Statistical stability} we show 
the statistical stability discussed above, and in Section \ref{sec:Reduction to the deterministic inverse problem} we reduce the proof of Theorem \ref{th_main} to that of Theorem \ref{th:main_geom}.
We prove Theorem \ref{th:main_geom} in Section \ref{sec_deterministic}. 
For the convenience of the reader, we have collected some well-known results in an appendix.  

\subsection{Previous literature}

For previous mathematical results on passive imaging problems we refer to \cite{garnier2009passive, de2013retrieval}.
The monograph \cite{garnier2016passive} gives a thorough review of the related literature. 
Passive imaging problems arise in geophysical applications. In seismic imaging ambient noise sources, that appear due to nonlinear interaction of ocean waves with the ocean bottom, can be utilized to image the wave speed in the subsurface of the Earth, see e.g. \cite{shapiro2005high,stehly2006study,yao2006surface}.

We also mention the closely related topic of imaging random media by time reversal techniques
\cite{borcea2002imaging, borcea2003theory, bal2003time, de2009estimating} as well as inverse scattering from random potential or random boundary conditions \cite{caro2016inverse, lassas2008inverse, helin2014inverse}.

Let us now turn to results on deterministic inverse problems similar to Theorem \ref{th:main_geom}.
In such coefficient determination problems,
it is typical to use
the Dirichlet-to-Neumann map to model the data.
Apart from immediate applications, this is reasonable since
several other types of data can be reduced to the 
Dirichlet-to-Neumann case. For instance, in \cite{Nachman1988a}
an inverse scattering problem is solved via a reduction to 
the inverse conductivity problem in \cite{Sylvester1987}, and the latter uses the Dirichlet-to-Neumann map as data. 
In the present paper, however,
we do not perform a reduction to the Dirichlet-to-Neumann case but adapt techniques originally developed in that case to the case of local source-to-solution map $\Lambda_\chi$.

The approach that we use is a modification of the Boundary Control method. This method was first developed by Belishev to the acoustic wave equation on $\R^n$ with an isotropic wave speed
\cite{belishev1987approach}. A geometric version of the method, suitable when the wave speed is given by a Riemannian metric tensor as in the present paper, was introduced by Belishev and Kurylev \cite{belishev1992reconstruction}.
We refer to \cite{Katchalov2001} for a thorough review of the related literature. 
Local reconstruction of the geometry from the local source-to-solution map $\Lambda_\chi$ has been studied as a part of iterative schemes, see e.g. \cite{Isozaki2010,Kurylev2015}. In the present paper we give a global uniqueness proof that does not rely on an iterative scheme. 
For general aspects of unique solvability in geometric inverse problems, see \cite{uhlmann1998inverse,krupchyk2008inverse,de2014reconstruction,lassas2014inverse} and references therein.

We restrict our attention to the unique solvability of the inverse problem but note that several variants of the Boundary Control method have been studied computationally \cite{belishev1999dynamical,Hoop2016,kabanikhin2004direct,pestov2010numerical} 
and stability questions have been investigated \cite{anderson2004boundary,korpela2016regularization,Liu2012}.

This work continues the line of research started by the authors in \cite{HLO1,HLO2}, where similar unique solvability of the geometry was considered for random and pseudo-random boundary sources. A novel feature of this paper is that we consider passive imaging, when the source is not assumed to be known.  

\section{The stochastic direct problem}
\label{sec:direct_problem}

In this section we show that the running averages $C_T$, $T > 0$, are well-defined as random variables.
Let us first recall the concept of generalized Gaussian random variable \cite{GV4}. 
A cylindrical set in a locally convex vector space $V$ with the dual $V'$ is 
of the form
\begin{equation*}
	\left\{u \in V \; | \; \left(\langle \ell_1,u\rangle, \ldots, \langle \ell_k, u\rangle\right) \in B \right\},
\end{equation*}
where $k\geq 1$, $\ell_1,\ldots,\ell_k \in V'$, and $B$ is a Borel subset of $\R^
k$, i.e., $B\in{\mathcal B}(\R^k)$. Above, we write $\langle \cdot,\cdot\rangle = \langle \cdot,\cdot\rangle_{V'\times V}$ for the dual pairing between $V'$ and $V$.
The $\sigma$-algebra generated by cylindrical sets in $V$ is denoted by ${\mathcal B}_c(V)$.
Notice that the cylindrical $\sigma$-algebra is always a subset of the Borel $\sigma$-algebra, and the two $\sigma$-algebras
are known to coincide if $V$ is a separable Fr\'{e}chet space \cite[Thm. A.3.7.]{Bogachev}.

We denote the rapidly decaying functions on $\R^n$ by ${\mathcal S}(\R^d)$.
The topological dual of ${\mathcal S}(\R^d)$ is the space of tempered distributions ${\mathcal S}'(\R^d)$.
It is well-known that ${\mathcal S}'(\R^d)$ is a locally convex topological vector space (even nuclear).

Throughout the paper, let $(\Omega, {\mathcal F}, \P)$ stand for a complete probability space.

\begin{definition}
A generalized random variable is a measurable function
\begin{equation*}
	X : (\Omega, {\mathcal F}) \to ({\mathcal S}'(\R^d), {\mathcal B}_c({\mathcal S}'(\R^d))).
\end{equation*}
A generalized random variable $X$ is called Gaussian, if for all $\phi_1,\ldots,\phi_k \in {\mathcal S}(\R^d)$, $k\in \N$, the mapping
\begin{equation*}
	\Omega \ni \omega \mapsto \left(\langle X(\omega),\phi_1\rangle,\ldots,\langle X(\omega),\phi_k\rangle\right) \in \R^k
\end{equation*}
is a Gaussian random variable.
\end{definition}

The probability law of a generalized Gaussian random variable $X$ is determined by the expectation $\expec X$ and the covariance operator $C_X : {\mathcal S}(\R^d) \to {\mathcal S}'(\R^d)$ defined by
\begin{equation}
	\label{eq:general_covariance}
	\langle \psi_1, C_X \psi_2\rangle = \expec \left(\langle X-\expec X, \psi_1\rangle \langle X-\expec X, \psi_2\rangle \right).
\end{equation}
If $X$ is zero-mean and satisfies $C_X = \iota$, where $\iota : {\mathcal S}(\R^d) \to {\mathcal S}'(\R^d)$ is the identity operator $\iota(\phi)=\phi$,
then $X$ is called Gaussian white noise.

\begin{remark}
The construction above is identical for generalized random variables obtaining values in the space of generalized functions ${\mathcal D}'(\R^d)$. This was also the original formulation in \cite{GV4}.
\end{remark}

It was proved by Kusuoka in \cite{Kusuoka} that for any $\epsilon>0$, white noise satisfies 
\begin{equation}
	\label{eq:w_sobolev_weighted}
	\W \in H^{-d/2-\epsilon}(\R^d; \langle x\rangle^{-d/2-\epsilon}) \quad \textrm{almost surely},
\end{equation}
where the weight function is defined by $\langle x\rangle = (1+|x|^2)^{1/2}$.  
Moreover, we have $H^{-d/2-\epsilon}(\R^d; \langle x\rangle^{-d/2-\epsilon}) \in {\mathcal B}_c({\mathcal S}'(\R^d))$ (see e.g. \cite[Prop. 7]{Fageot}) and therefore we can consider $\W$ as a random variable restricted to 
$H^{-d/2-\epsilon}(\R^d; \langle x\rangle^{-d/2-\epsilon})$ assigned with the cylindrical $\sigma$-algebra. Since the weighted Sobolev space is separable (and Fr{\'e}chet), the cylindrical $\sigma$-algebra coincides with the Borel $\sigma$-algebra and $\W$ is Borel measurable in $H^{-d/2-\epsilon}(\R^d; \langle x\rangle^{-d/2-\epsilon})$.
Finally, since we have a continuous embedding 
$H^{-d/2-\epsilon}(\R^d; \langle x\rangle^{-d/2-\epsilon}) \subset H^{-d/2-\epsilon}_{loc}(\R^d)$, we can identify $\W$ as a random variable
\begin{equation*}
	\W : (\Omega, {\mathcal F}) \to (H^{-d/2-\epsilon}_{loc}(\R^d), {\mathcal B}(H^{-d/2-\epsilon}_{loc}(\R^d))).
\end{equation*}

We denote by $\Box_\chi^{-1}$ the solution operator of (\ref{eq:problem}),
that is, $\Box_\chi^{-1} (W) = u$ where $u$ solves (\ref{eq:problem}) and $u$ is defined to be zero for negative times.
Then 
$$
\Box_\chi^{-1} : H_{loc}^\sigma(\R^{1+n}) \to H^{\sigma+1}_{loc}(\R^{1+n}), \quad \sigma \in \R,
$$
is continuous, see e.g. \cite[Thm. 23.2.4]{Hormander3}.
We denote by $\tau^s$ the translation by $s \in \R$ in time, that is,
$$
\tau^s \phi(t) = \phi(t + s), \quad \phi \in C_0^\infty(\R),
$$
and extend this definition to $\D'(\R)$ by 
$$
\bra \tau^s w, \phi \cet_{\D'\times C_0^\infty(\R)} = \bra w, \tau^{-s} \phi \cet_{\D'\times C_0^\infty(\R)}.
$$
The function
$$
\Phi : (0, T )\times \R \to \R, \quad  \Phi(s, t) = \tau^{−s} \phi(t)
$$
is smooth, and moreover $\Phi = 0$ when $t \notin (0, T+R)$ where $R > 0$ is such that $\supp(\phi) \subset (0,R)$. Hence function
\begin{equation}
	\label{eq:translated_duality}
s \mapsto \pair{w, \Phi(s, \cdot)}_{\D'\times C_0^\infty(\R)} = \pair{\tau^s w, \phi}_{\D'\times C_0^\infty(\R)}
\end{equation}
is smooth for all $w \in \D'(\R)$  and $\phi \in C^{\infty}_0(\R)$, see \cite[Thm. 2.1.3]{Hormander1}. An analogous argument shows that
\begin{equation*}
s \mapsto \pair{\tau^s w \otimes \tau^s w, \phi}_{\D'\times C_0^\infty(\R^{2+2n})}
\end{equation*}
is smooth for all $w \in \D'(\R^{1+n})$ and $\phi \in C^{\infty}_0(\R^{2+2n})$. Here $\otimes$ denotes
the tensor product of distributions, see e.g. \cite[Thm. 5.1.1]{Hormander1}  for the definition.

For a fixed $T > 0$, we define the map 
$$
A_T(w) = \frac 1 {T} \int_0^T \tau^s w \otimes \tau^s w\, ds, \quad w \in H^{\sigma}_{loc}(\R^{1+n}),
$$
in the sense of the Pettis integral, that is,
$$
\bra A_T(w), \phi \cet_{\D'\times C_0^\infty(\R^{2+2n})}
= \frac 1 {T} \int_0^T \bra \tau^s w \otimes \tau^s w, \phi \cet_{\D'\times C_0^\infty(\R^{2+2n})} ds.
$$

The integral above defines $A_T(w)$ as a generalized function in $\D'(\R^{2+2n})$ and, moreover, yields a continuous map in the following sense:
\begin{lemma}
The map $A_T : H^{-\sigma}_{loc}(\R^{1+n}) \to H^{-\sigma}_{loc}(\R^{2+2n})$, $\sigma \in \R$, is continuous.
\end{lemma}
\begin{proof}
We recall that the topology of $H^{-\sigma}_{loc}(\R^{1+n})$ is induced by the semi-norms
$$
w \mapsto \norm{\psi w}_{H^{-\sigma}(\R^{1+n})},  
\quad \psi \in C_0^\infty(\R^{1+n}).
$$
Let $w_0 \in H^{-\sigma}_{loc}(\R^{1+n})$, $\psi \in C_0^\infty(\R^{2+2n})$ and $\epsilon > 0$.
In order to show that $A_T$ is continuous, it is enough to show \cite[p. 64]{Treves} that there are $\tilde \psi \in C_0^\infty(\R^{1+n})$ and $\delta > 0$ such that 
$$
\norm{\tilde \psi (w-w_0)}_{H^{-\sigma}(\R^{1+n})} < \delta
\quad \text{implies} \quad 
\norm{\psi (A_T(w)-A_T(w_0))}_{H^{-\sigma}(\R^{2+2n})} < \epsilon.
$$
We choose $\tilde \psi \in C_0^\infty(\R^{1+n})$
so that $(\tilde \psi \otimes \tilde \psi) \tau_1^{-s} \tau_2^{-s} \psi = \tau_1^{-s} \tau_2^{-s} \psi$ for all $s \in (0,T)$. Here $\tau_j^{-s}$, $j=1,2$, act in the different time variables. Let $\phi \in H^{\sigma}(\R^{2+2n})$. It follows that
\begin{align*}
&|\bra \psi (A_T(w) - A_T(w_0)), \phi \cet_{H^{-\sigma} \times H^{\sigma}(\R^{2+2n})}|
\\&\quad\le 
\frac 1 {T} \int_0^T \norm{(\tilde \psi \otimes \tilde \psi)(w \otimes w - w_0 \otimes w_0)}_{H^{-\sigma}(\R^{2+2n})} \norm{\tau_1^{-s} \tau_2^{-s} \psi \phi}_{H^{\sigma}(\R^{2+2n})} ds
\\&\quad\le C
\norm{\tilde \psi (w-w_0) \otimes \tilde \psi w + \tilde \psi w_0 \otimes \tilde \psi(w -w_0)}_{H^{-\sigma}(\R^{2+2n})} \norm{\phi}_{H^{\sigma}(\R^{2+2n})}.
\end{align*}
Finally, for small $\delta > 0$
\begin{align*}
&\norm{\tilde \psi (w-w_0) \otimes \tilde \psi w + \tilde \psi w_0 \otimes \tilde \psi(w -w_0)}_{H^{-\sigma}(\R^{2+2n})}
\\&\quad\le \delta  \norm{\tilde \psi w}_{H^{-\sigma}(\R^{1+n})}
+ \norm{\tilde \psi w_0}_{H^{-\sigma}(\R^{1+n})} \delta
\le C \delta.
\end{align*}
\end{proof}

By combining the continuity results above, we define $C_T(\omega) = A_T(\Box_\chi^{-1} (\mathbb W(\omega)))$, $T > 0$, and see that 
$$
C_T : \Omega \to (H^{\sigma}_{loc}(\R^{2+2n}),{\mathcal B}(H^{\sigma}_{loc}(\R^{2+2n}))) ,
\quad \sigma < -\frac{1+n}2 + 1,
$$
is a random variable. 
\begin{remark}
Since the weighted Sobolev space $H^{-d/2-\epsilon}(\R^d; \langle x\rangle^{-d/2-\epsilon})$ is separable, the random variable $\W$ in \eqref{eq:w_sobolev_weighted} has the Radon property \cite{Bogachev}.
Notice carefully that the Radon property is transferred through any continuous mappings and therefore also $C_T$ is Radon.
\end{remark}
%
%
%

\section{The stochastic inverse problem and statistical stability}
\label{sec:Statistical stability}

For any function $f \in C_0^\infty(\R^{1+n})$, let us define $v^f = v$ as the solution of a time reversed wave equation 
\begin{align}
\label{eq_wave_test}
&\p_t^2 v - \Delta_g v = f \quad \text{in $(-\infty,S) \times \R^n$},
\\
&v|_{t = S} = \p_t v|_{t= S} = 0,\nonumber
\end{align}
where $S \in \R$ is large enough so that $f \in C_0^\infty((-\infty, S) \times \R^{n})$. 
In this section we show the following theorem. 

\begin{theorem}
\label{th_inner_prods_by_ergodicity}
Suppose that $n \ge 3$, $(\R^n,g)$ is non-trapping and that $g$ coincides with the Euclidean metric outside a compact set. 
Let $\mathbb D \subset C_0^\infty((0,\infty) \times \Src)$
be a countable set. 
There exists $\Omega_0 \subset \Omega$ such that $\P(\Omega_0) = 0$
and for all $\omega \in \Omega \setminus \Omega_0$ and all $f,h \in \mathbb D$, it holds that
$$
\lim_{T \to \infty} \bra C_T(\omega), f \otimes h \cet_{\D'\times C_0^\infty(\R^{2+2n})}
= \bra \kappa v^f, \kappa v^h\cet_{L^2(\R^{1+n})}.
$$
\end{theorem}

In what follows, we write $\bra \cdot, \cdot\cet = \bra \cdot, \cdot \cet_{\D'\times C_0^\infty(\R^{2+2n})}$.
\begin{lemma}
\label{lem:inner poduct of white noise}
Let $W \in \D'(\R^{1+n})$ and $f \in C_0^\infty(\R^{1+n})$ be arbitrary sources in problems \eqref{eq:problem} and \eqref{eq_wave_test}, respectively. Moreover, let $u$ and $v^f$ be the corresponding solutions.
Then we have the identity
$$
\bra u, f \cet = \bra W, \chi v^f \cet.
$$

\end{lemma}
\begin{proof}
Suppose that $W \in C_0^\infty(\R^{1+n})$.
The general case follows since test functions are dense in distributions. Next, let $v$ and $S$ be as in (\ref{eq_wave_test}).
Using the shorthand notation $\Box_g = \p_t^2 - \Delta_g$, we have that
\begin{align*}
\bra u, f \cet
=
\bra u, \Box_g v \cet_{L^2((0,S) \times \R^n)}
= 
\bra \Box_g u, v \cet_{L^2((0,S) \times \R^n)}
=
\bra W, \chi v \cet.
\end{align*}
This proves the claim.
\end{proof}

Let us recall the following result regarding the local energy decay which is due to Vainberg \cite{vainberg75,vainberg_book},
see \cite{vodev} for the formulation as below.
\begin{theorem}
\label{thm:led}
Let $u \in C^\infty((0,\infty) \times \R^{n})$ solve the problem
\begin{align*}
&\p_t^2 u - \Delta_g u = 0, \quad \text{in $(0,\infty) \times \R^{n}$},
\\
&u|_{t=0} = u_0, \quad \p_t u|_{t=0} = u_1.
\end{align*}
Suppose that $u_0$ and $u_1$
are compactly supported.
Suppose that $(\R^n,g)$ is non-trapping and that $g$ coincides with the Euclidean metric outside a compact set. 
Then there is $t_0 > 0$ such that $u$ satisfies
local energy decay
\begin{equation*}
	\int_{\R^n} \left( |\partial_t u(t,x)|^2 + |\nabla u(t,x)|^2 \right) \chi(x) dx \leq C \eta(t) E_0,	
	\quad t > t_0,
\end{equation*}
for any compactly supported function $\chi \in C_0^\infty(\R^n)$. Here we have
\begin{align*}
E_0 &= \int_{\R^{n}} |\nabla u_0(x)|^2  + |u_1(x)|^2 dx,
\quad
\eta(t) =
\begin{cases} 
e^{-b t}, & n\geq 3 \textrm{ odd}, \\
t^{-2n}, & n\geq 2 \textrm{ even}, 
\end{cases}
\end{align*}
and the constants $C,b>0$ depend on $g$, $\chi$ and the supports of $u_0$ and $u_1$.
\end{theorem}

We need a decay estimate for the norm $\norm{u(t,\cdot)}_{L^2(K)}$
where $K \subset \R^n$ is compact.

\begin{lemma}
\label{lem:led}
Let $(\R^n,g)$ be as in Theorem \ref{th_inner_prods_by_ergodicity}
and let $u$ be as in Theorem \ref{thm:led}. 
Let $K \subset \R^n$ be compact.
Then there is $t_0 > 0$ such that $u$ satisfies
\begin{equation*}
\norm{u(t,\cdot)}_{L^2(K)} \le C \mu(t) E_0, \quad t > t_0,
\end{equation*}
where
\begin{align}
\label{eq:myy}
\mu(t) =
\begin{cases} 
e^{-b t}, & n\geq 3 \textrm{ odd}, \\
t^{-2n+1}, & n\geq 4 \textrm{ even}, 
\end{cases}
\end{align}
\end{lemma}
\begin{proof}
To simplify the notation, we assume without loss of generality that $E_0 =1$. Let $B(r) = \{\|x\| < r\}$ be the Euclidean ball of radius $r$ and write
$$
u_r(t) = \frac 1{|B(r)|} \int_{B(r)} u(t,x) dx,
$$
where $|B(r)|$ is the volume of $B(r)$.
Theorem \ref{thm:led} implies
$
|\p_t u_r(t)| \le C \eta(t) 
$
where the constant $C > 0$ depends on $r > 0$ and $g$.
Thus for $t_0 < t < s$,
\begin{align}
\label{ed_step1}
|u_r(t) - u_r(s)| \le C \int_t^s \eta(\tau) d\tau = C (\mu(t) - \mu(s)).
\end{align}
We see that $\lim_{t \to \infty} u_r(t)$
exists, and denote the limit by $\bar u(r)$.

The Poincar\'e-Wirtinger  inequality
\begin{align*}
\|u(t,\cdot)-u_r(t)\|_{L^2(B(r))} \leq C \|\nabla u(t,\cdot)\|_{L^2(B(r))},
\end{align*}
together with Theorem \ref{thm:led} and (\ref{ed_step1}), implies that 
\begin{align}
\label{ed_step2}
\|u(t,\cdot)- \bar u(r)\|_{L^2(B(r))}
\le C \eta(t) + |u_r(t)- \bar u(r)|\|1\|_{L^2(B(r))}
\le C \mu(t).
\end{align}
In particular, for $0 < r_1 < r_2$, 
$u(t,\cdot) \to \bar u(r_j)$, $j=1,2$, in $L^2(B(r_1))$.
Thus $\bar u(r)$ does not depend on $r > 0$ and we denote it by $\bar u$.

It remains to show that $\bar u = 0$.
As $u(t)$ is compactly supported, by the finite speed of propagation,
the Gagliardo-Nirenberg-Sobolev inequality implies that
$$
\norm{u(t,\cdot)}_{L^{p*}(\R^n)} \le C \norm{\nabla u(t,\cdot)}_{L^2(\R^n)},
$$
where $p^*$ is the Sobolev conjugate of $2$, that is, 
$1/p^* = 1/2 - 1/n$.
Note that $p^* > 2$. We apply H\"older's inequality with $p=p^*/2$
and $1/p+1/q=1$,
$$
\int_{B(r)} u^2(t,\cdot) dx \le \norm{u^2(t,\cdot)}_{L^p(B(r))} \norm{1}_{L^q(B(r))}.
$$
The conservation of energy implies that $\norm{\nabla u(t,\cdot)}_{L^2(\R^n)}$, $t > 0$, is bounded. Thus
$\norm{u(t,\cdot)}_{L^2(B(r))}^2 \le C r^{n/q}$
with a constant $C > 0$ independent of $r$.

To get a contradiction, suppose now that $\bar u \ne 0$. Then there is $\epsilon > 0$ such that 
$$\norm{\bar u}_{L^2(B(r))}^2 =\bar{u}^2 \norm{1}_{L^2(B(r))}^2= 2 \epsilon r^n.$$
By the convergence (\ref{ed_step2}), 
for all $r > 0$ there is $t_r$ such that $\norm{u(t_r,\cdot)}_{L^2(B(r))}^2 \ge \epsilon r^n$.
Thus $r^{n-n/q} \le C$, $r > 0$, which is a contradiction since $q > 1$. 
\end{proof}

\begin{lemma}
\label{lem:led2}
Let $(\R^n,g)$ be as in Theorem \ref{th_inner_prods_by_ergodicity}.
Suppose that $K \subset \R^n$ is compact and $f\in C^\infty_0(\R^n)$. Let $u \in C^\infty((0,\infty) \times \R^{n})$ solve the problem
\begin{align*}
&\p_t^2 u - \Delta_g u = f, \quad \text{in $(0,\infty) \times \R^{n}$},
\\
&u|_{t=0} = \p_t u|_{t=0} =0.
\end{align*}
Then there exists $t_0>0$ such that for all $t>t_0$ 
$$
\|u(t,\cdot)\|_{L^2(K)}\leq C \mu(t)\|f\|_{L^2(\R^{1+n})},
$$
where $\mu(t)$ is defined in \eqref{eq:myy}.
Here the constants $C$ and $t_0$ depend on $g$, $K$ and the support of $f$.
\end{lemma}

\begin{proof}
Let $t_1>0$ be such that supp$(f)\subset [0,t_1]\times \R^n$. By the finite speed of wave propagation, it holds that supp$(u|_{t=t_1})$ and supp$(\p_tu|_{t=t_1})$ are compact in $\R^n$. Consider the solution $v$ of the initial value problem 
\begin{align*}
&\p_t^2 v - \Delta_g v = 0, \quad \text{in $(t_1,\infty) \times \R^{n}$},
\\
&v|_{t=t_1} =u(t_1),\: \p_t v|_{t=t_1} =\p_tu|_{t=t_1}.
\end{align*}
By the uniqueness, it must hold that $v=u$. By Lemma \ref{lem:led} there exists $t_0>t_1$ and constant $C$ independent of $t> t_0$ such that 
\begin{equation*}
\norm{u(t,\cdot)}_{L^2(K)} \le C \mu(t) E_0, \quad t > t_0,
\end{equation*}
Where $E_0=\int_{\R^{n}} |\nabla u(t_1,\cdot)|^2  + |\p_t u(t_1,\cdot)|^2 dx$. As $u$ is an energy class solution of a wave equation with zero initial values, by the standard energy estimates for the wave equation it holds that
$$
E_0\leq C\|f\|^2_{L^2(\R^{1+n})}.
$$
This proves the claim.
\end{proof}

\begin{lemma}
\label{lem:conv_of_expec}
Let $(\R^n,g)$ be as in Theorem \ref{th_inner_prods_by_ergodicity}.
Let $S>0$ and $f,h \in C_0^\infty((0,S) \times \R^n)$. It follows that
\begin{equation*}
	\lim_{T \to \infty} \expec \bra C_T, f \otimes h\cet
	= \bra \kappa v^f, \kappa v^h\cet_{L^2((-\infty, S) \times \R^n)}.
\end{equation*}
\end{lemma}
\begin{proof}
Here we will use notation $f^s(t,x)=f(t+s,x)$ for a time sift $s\in \R$. By the Lemma \ref{lem:inner poduct of white noise} and standard energy estimates, we have
\begin{equation*}
	\expec \bra u^s, f\cet^2 = \expec \bra \W^s, \chi^s v^f\cet^2 = \bra \chi^s v^f, \chi^s v^f\cet \leq C \norm{f}^2_{L^2(\R^{1+n})}, \quad s<T,
\end{equation*}
where the constant $C$ depends on $T$.
Therefore, we see that the mapping
\begin{equation*}
	(\omega, s) \to  \bra u^s(\omega), f\cet \bra u^s(\omega), h\cet
\end{equation*}
is integrable on $\Omega \times (0,T)$ with respect to $\P \times {\rm d}t$.
In consequence, together with \eqref{eq:general_covariance} the Fubini theorem yields
\begin{equation*}
	\expec \bra C_T, f \otimes h \cet
	= \frac 1 T \int_0^T \expec \bra \chi^s \W^s, v^f\cet \bra \chi^s \W^s, v^h\cet ds 
	= \frac 1 T \int_0^T \bra \chi^s v^f,  \chi^s v^h\cet ds.
\end{equation*}
For the time-shifted characteristic function we have 
$$
\chi^s(t,x) = \chi_0^s(t) \kappa(x) = \kappa(x) - (1-\chi_0^s(t))\kappa(x)
$$
and $\supp(1-\chi_0^s) \subset (-\infty, 1-s)$.
By the local energy decay in Lemma \ref{lem:led2}, there is a constant $C > 0$
depending on $g$ and the supports of $\kappa$ and $f$ such that 
$$
\norm{v^f}_{L^2((-\infty, 1-s) \times \supp(\kappa))} 
\leq C \left(\int_{-\infty}^{1-s} |t|^{-4n+2} dt\right)^{1/2} \norm{f}_{L^2(\R^{1+n})}
\leq C s^{-2n+\frac 32} \norm{f}_{L^2(\R^{1+n})},
$$
for large $s$.
Hence we obtain
$$
\expec \bra C_T, f \otimes h \cet = \bra \kappa v^f,  \kappa v^h\cet + \frac 1 T \int_0^T R(s) ds,
$$
where 
$$
|R(s)| 
\le C s^{-2n+\frac 32} \norm{f}_{L^2(\R^{1+n})} \norm{h}_{L^2(\R^{1+n})}
$$
To conclude, one has
$$
\frac 1 T \int_0^T |R(s)| ds \le C T^{-2n+\frac 32}
$$
and the claim follows.
\end{proof}

In order to show the statistical stability of the data, we need the following result from ergodic theory (see e.g. \cite[p. 94]{CramerLeadbetter}):

\begin{theorem}\sl
\label{thm:aux_ergodicity}
Let $\widetilde{Z}_t$, $t\geq 0$, be a real-valued random variables such that $\expec \widetilde{Z}_t = 0$ and the covariance function 
$(t,s) \mapsto \expec (\widetilde{Z}_t \widetilde{Z}_{s})$, $t,s\geq 0$, is continuous. Assume that for some constants $c,\epsilon>0$ the condition
\begin{equation*}
	|\expec (\widetilde{Z}_t \widetilde{Z}_{t+r})| \leq c(1+r)^{-\epsilon}
\end{equation*}
holds for all $t \geq 0 $ and $r \geq 0$. Then,
\begin{equation*}
	\lim_{T\to\infty} \frac 1 {T} \int_{0}^T \widetilde{Z}_t \, dt = 0 \quad \text{almost surely}.
\end{equation*}
\end{theorem}

\begin{lemma}
\label{lem:variance_decay}
Let $(\R^n,g)$ be as in Theorem \ref{th_inner_prods_by_ergodicity}.
Let $f,h \in C_0^\infty((0,S) \times \R^n)$ and use notation
\begin{equation*}
	Z_r = \bra u^r, f\cet \bra u^r, h\cet
\end{equation*}
 Then
there is $C > 0$ depending on $n$, $g$, and the supports of $\kappa$, $f$ and $h$
such that 
\begin{equation*}
			|\expec (Z_r -\expec Z_r)(Z_{r+s}-\expec Z_{r+s})| \leq C(1+s)^{-n}\norm{f}_{L^2(\R^{1+n})}^2 \norm{h}_{L^2(\R^{1+n})}^2.
\end{equation*}
\end{lemma}

\begin{proof}
For convenience, let us write $X^r = \bra u^r, f\cet$ and $Y^r = \bra u^r, h\cet$.
By the Isserlis formula \cite{isserlis} for Gaussian random variables we have
\begin{equation*}
	\expec Z_r Z_{r+s} = \expec X^r Y^r \expec X^{r+s} Y^{r+s}
	+ \expec X^r X^{r+s} \expec Y^r Y^{r+s} + \expec X^r Y^{r+s} \expec Y^r X^{r+s}
\end{equation*}
and, consequently,
\begin{equation}
	\label{eq:variance_id}
	\expec (Z_r -\expec Z_r)(Z_{r+s}-\expec Z_{r+s})
= \expec X^r X^{r+s} \expec Y^r Y^{r+s} + \expec X^r Y^{r+s} \expec Y^r X^{r+s}.
\end{equation}
We write $v_r^f(t,\cdot) = v^f(t-r,\cdot)$.
The local energy decay, Lemma \ref{lem:led2}, implies
\begin{eqnarray}
	\label{eq:cor_energy_est}
	|\expec X^r Y^{r+s}| & = & |\expec \bra \chi^r \W^r , v^f\cet \bra \chi^{r+s} \W^{r+s}, v^h\cet| \nonumber\\
	& = & |\expec \bra \chi \W, v^f_r\cet \bra \chi \W, v^h_{r+s}\cet| \nonumber\\
	& = & |\bra \chi v^f_r, \chi v^h_{r+s}\cet| \nonumber\\
	& \leq & C (1 + s)^{-2n+\frac 32}\norm{f}_{L^2(\R^{1+n})} \norm{h}_{L^2(\R^{1+n})},
\end{eqnarray}
since $v_r^f$ is small in $\supp(v_{r+s}^h)$
for $s \gg 0$. 
\end{proof}

\begin{proof}[Proof of Theorem \ref{th_inner_prods_by_ergodicity}]
For a fixed pair of sources $(f,h)$ we set $\widetilde{Z}_t = Z_t - \expec Z_t$, where $Z_t = \bra u^r, f\cet \bra u^r, h\cet$.
Continuity of the covariance function of $Z_t$ follows by considering equality \eqref{eq:variance_id}.
Note that the correlations between $X^r, X^{r+s}, Y^r$ and $Y^{r+s}$ on the right hand side of \eqref{eq:variance_id}
are all represented by inner products between smooth functions in the spirit of \eqref{eq:cor_energy_est}. Since these inner products are smooth functions with respect to $r$ and $s$, it follows that the covariance function in \eqref{eq:variance_id} is continuous. 
Next, we combine Lemma \ref{lem:conv_of_expec} and Lemma \ref{lem:variance_decay} to validate Theorem \ref{thm:aux_ergodicity}. As a countable set of source pairs (countable union of zero measurable sets is zero measurable, the claim follows for all $(f,h) \in \mathbb D$.
\end{proof}

We conclude this section with the following simple lemma to quantify the convergence of the data. Notice that Lemma \ref{lem:var_estimate} is not needed
for the previous proof.

\begin{lemma}
\label{lem:var_estimate}
Let $f,h \in C_0^\infty((0,S) \times \R^n)$. Then
there is $C > 0$ depending on $n$, $g$, and the supports of $\kappa$, $f$ and $h$
such that 
\begin{equation*}
{\rm Var} \bra C_T, f \otimes h \cet  \leq 
C T^{-2} \norm{f}_{L^2(\R^{1+n})}^2 \norm{h}_{L^2(\R^{1+n})}^2
\end{equation*}
\end{lemma}
\begin{proof}
In the proof of Lemma \ref{lem:conv_of_expec} we showed that the Gaussian random variables $X^r$ and $Y^r$ have a bounded variance independent of $r$. Since any moment of a Gaussian random variable is bounded by a constant depending on the variance, we see that the mapping
\begin{equation*}
	(\omega, r,s) \to  X^r Y^r X^s Y^s
\end{equation*}
is integrable over $\Omega \times (0,T) \times (0,T)$ for any fixed $T>0$ with respect to $\P \times {\rm d} r \times {\rm d} s$.

Now the Fubini theorem yields that 
\begin{eqnarray*}
	\expec \bra C_T, f\otimes h\cet^2  & = & \frac 1 {T^2} \int_0^T \int_0^T \expec X^s Y^s X^r Y^r ds dr \quad {\rm and} \\
	(\expec \bra C_T, f\otimes h\cet)^2 & = & \frac 1 {T^2} \int_0^T \int_0^T \expec X^s Y^s \expec X^r Y^r ds dr.
\end{eqnarray*}
It follows by equation \eqref{eq:variance_id} and estimate \eqref{eq:cor_energy_est} that
$$
{\rm Var}(\bra C_T, f\otimes h\cet)
\le C \norm{f}_{L^2(\R^{1+n})}^2 \norm{h}_{L^2(\R^{1+n})}^2 \frac 1 {T^2} \int_0^T \int_0^T
(1 + |r-s|)^{-4n+3} ds dr
$$
and the claim follows by estimating the the double integral in time by
\begin{eqnarray*}
	\int_0^T \int_0^T (1 + |r-s|)^{-4n+3} ds dr
	& = & \frac 12 \int_0^{2T} \int_{2T-s'}^{s'} (1+r')^{-4n+4} dr' ds' \\
	& = & \frac 1{2(1-n)} \int_0^{2T}\left((1+s')^{1-n}-(1+2T-s')^{1-n}\right) ds' \\
	& \leq & C(1+T^{5-4n}) \leq C.
\end{eqnarray*}
for any $n\geq 3$.
\end{proof}

\section{Reduction to the deterministic inverse problem}
\label{sec:Reduction to the deterministic inverse problem}

In this section we will show the following theorem.
\begin{theorem}
\label{th_response}
Let  $\mathbb{D} \subset  C_0^\infty((0,\infty) \times \Src)$ be a dense set, and consider the data
\begin{align}
\label{inner pro data}
\pair{\kappa v^f, \kappa v^h}_{L^2(\R^{1+n})}, 
\quad f,h \in \mathbb{D},
\end{align}
where functions $v^f(t,x)$ and $v^h(t,x)$ solve (\ref{eq_wave_test}) with the sources $f(t,x)$ and $h(t,x)$, respectively. 
Then the data \eqref{inner pro data} determine the local source-to-solution map $\Lambda_{\Src}$ as defined in Theorem \ref{th:main_geom}.
\end{theorem}

It follows from the assumptions in Theorem \ref{th_main} that the Riemannian manifold $(\R^n,g)$ is complete. Indeed, the metric tensor $g$ coincides with the Euclidean metric $e$ outside a compact set, and therefore there exist uniform constants $c,C>0$ such that $c\|\cdot\|_e\leq \|\cdot\|_g\leq C \|\cdot\|_e$, where $\|\cdot\|_e$ stands for the Euclidean and $\|\cdot\|_g$ for the Riemannian norm.  
Thus Theorems \ref{th_inner_prods_by_ergodicity}, \ref{th_response}
and \ref{th:main_geom} imply Theorem \ref{th_main}.

We will prove two auxiliary lemmas before presenting a proof for Theorem \ref{th_response}. 
Let $d_0$ be the Euclidean distance in $\R^{n}$,
and denote by $d_g$ the Riemannian distance in $(\R^{n},g)$.
For $p\in \R^n$ and $r>0$, we denote the respective open balls
by $B_{0}(p,r)$ and $B_{g}(p,r)$.
We will use the shorthand notation $\square_g = \p_t^2 - \Delta_g$.
  
\begin{definition}
For $\mathcal B \subset (0, \infty) \times \R^n$,
we say that $f\in C^\infty_0(\mathcal{B})$ is non-radiating, if ${\rm supp}(w^f) \subset \overline{\mathcal{B}}$ for the solution $w=w^f$ of 
\begin{align}
\label{eq_wave_test2}
&\p_t^2 w(t,x) - \Delta_g w(t,x) = f(t,x) \quad \text{in $(0, \infty) \times \R^n$},
\\
&w(0,x)| = \p_t w(0,x) = 0, \hbox{ for all }x\in N.\nonumber
\end{align}
Furthermore, we define
$
\mathcal N(\mathcal B) = \{f \in C^\infty_0(\mathcal{B}) \;|\;
\text{$f$ is non-radiating} \}.
$
\end{definition}

\begin{definition}
We define the future of a set $\mathcal{B} \subset \R^{1+n}$ by
\begin{eqnarray*}
{\mathcal I}^+(\mathcal{B}) & = & \left\{(t,x)\in \R^{1+n} \; | \; \textrm{there exists } (s, y) \in \mathcal{B} \textrm{ such that } t > s\right. \\ & & \left.\textrm{and } d_g(x,y) < t-s\right\}.
\end{eqnarray*}
\end{definition}

\begin{lemma}
\label{lem_nonrad_test}
Let $(t_0,x_0) \in \R \times \Src$, $\epsilon >0$, and define
$\mathcal B = (t_0-\epsilon,t_0) \times B_{0}(x_0,\epsilon)$, and $\mathcal Q = (t_0, t_0+1)\times \Src$.
Let $f \in C_0^\infty(\mathcal B)$.
For small $\epsilon>0$, 
$f\in \mathcal N(\mathcal B)$ if and only if
\begin{equation}
\label{non-radiating test}
	\bra \kappa w^f, \kappa w^h\cet_{L^2(\R\times N)} = 0,
\quad h\in C^\infty_0(\mathcal Q).
\end{equation}
Recall that $\kappa=\kappa(x)$ is independent of time.
\end{lemma}
\begin{proof}
Let $\epsilon >0$ be small enough so that 
\begin{equation}
\label{Light cone contained in S}
\mathcal{I}^+(\mathcal B)
\cap (\{t_0\}\times \R^n)\subset \{t_0\}\times \Src.
\end{equation} 
Clearly, $f \in \mathcal N(\mathcal B)$
implies (\ref{non-radiating test}).
Suppose now that (\ref{non-radiating test}) holds.
Let $\phi \in C^\infty_0(\mathcal Q)$. By choosing $h = \Box_g(\kappa^{-2} \phi)$, we have
$w^h = \kappa^{-2}\phi$ and further 
$$
\bra w^f, \phi\cet_{L^2(\R^{1+n})}
=\bra \kappa w^f, \kappa w^h\cet_{L^2(\R^{1+n})}=0.
$$
Thus $w^f = 0 $ in $\mathcal Q$. By \eqref{Light cone contained in S} and the finite speed of wave propagation, it holds that 
$w^f = 0$ in $(t_0, \infty)\times \R^n$. 
Using the finite speed of wave propagation once more, 
we see that $w^f(t,x) = 0$ when $t \in \R$ and
$d_g(x, B_0(x_0, \epsilon)) \ge \epsilon$.
The exterior domain $E := \R^n \setminus B_0(x_0, \epsilon)$ is connected
and $\p_t^2 w^f - \Delta_g w^f = 0$ in $\R \times E$.
Thus $w^f = 0$ in $\R \times E$ by unique continuation (Theorem \ref{Unique continuation} in the appendix). 
\end{proof}

\begin{lemma}
\label{lem:finding the distance function}
Let $x_1,x_2 \in \Src$, and let $\epsilon >0$ be so small that $B_{0}(x_j,\epsilon)\subset \Src$, $j=1,2$.
Let $t_0 > 0$ and define ${\mathcal C} = (0,\infty) \times B_{0}(x_1,\epsilon)$
and $\mathcal B = (t_0 - \epsilon,t_0) \times B_{0}(x_2,\epsilon)$.
Then 
\begin{equation}
\label{Future test for distances}
{\mathcal I}^+({\mathcal C}) \cap \mathcal B = \emptyset,
\end{equation}
if and only if 
\begin{equation}
\label{Future test for distances 2}
	\bra \kappa w^f, \kappa w^h\cet_{L^2(\R \times N)} = 0,
\quad h \in C^\infty_0({\mathcal C}),\ f \in \mathcal N(\mathcal B).
\end{equation}
\end{lemma}
\begin{proof}
As $f$ is non-radiating, 
the finite speed of wave propagation guarantees that 
(\ref{Future test for distances})
implies (\ref{Future test for distances 2}).
Suppose now that \eqref{Future test for distances} does not hold. 
The set $\mathcal{A}:={\mathcal I}^+({\mathcal C}) \cap \mathcal B$ 
is open and non-empty. Let $\phi\in C^\infty_0(\mathcal{A})$
be non-zero and $\phi\geq 0$. Choose $(s,x)\in \mathcal{A}$ such that $\phi(s,x) > 0$. By approximate controllability 
(Theorem \ref{density} in the appendix), there exists a source $h\in C_0^\infty(\mathcal{C})$ such that
$$
\bra \phi(s), w^h(s) \cet_{L^2(\R^n)}>0.
$$
Since $w^h$ and $\phi $ are continuous,
there is $\chi \in C_0^\infty(\R)$ such that 
$
\bra \chi \phi, w^h\cet_{L^2(\R^{1+n})}>0.
$ 
We define the function $f=\square_g (\kappa^{-2} \chi \phi)\in C^\infty_0(\mathcal B)$. Then $f \in \mathcal N(\mathcal B)$ and
$$
\bra \kappa w^f, \kappa w^h\cet_{L^2(\R \times N)}
=\bra \chi \phi, w^h\cet_{L^2(\R^{1+n})} > 0.
$$
Therefore \eqref{Future test for distances 2} is not valid either. \end{proof}

Now we are ready to present the proof of Theorem \ref{th_response}.
\begin{proof}[Proof of Theorem \ref{th_response}]
The inner products (\ref{inner pro data})
determine the same inner products for all $f,h \in C_0^\infty((0,\infty) \times \Src)$
by density. By reversing the time, these again determine the inner products
\begin{align}
\label{inner_prod_w}
\bra \kappa w^f, \kappa w^h\cet_{L^2(\R^{1+n})},
\quad f,h \in C_0^\infty((0,\infty) \times \Src).
\end{align}

Let $x_1, x_2 \in \Src$ and $\epsilon, t_0 > 0$
be as in Lemma \ref{lem:finding the distance function}.
Observe that 
\begin{align}
\label{eq_distance of sets}
d_g(B_0(x_1,\epsilon),B_0(x_2, \epsilon))=\sup\{t_0>0 \;|\;
\text{(\ref{Future test for distances 2}) is valid}\}
\end{align}
and
$
d_g(x_1,x_2)=\lim_{\epsilon \rightarrow 0} d_g(B_0(x_1,\epsilon),B_0(x_2, \epsilon)).
$
For $\mathcal B$ be as in Lemma \ref{lem_nonrad_test},
we can determine the set $\mathcal N(\mathcal B)$,
since the validity of (\ref{non-radiating test})
can be tested given the inner products (\ref{inner_prod_w}).
Thus also the validity of (\ref{Future test for distances 2})
can be tested given (\ref{inner_prod_w}),
and the distance function $d_g$ can be determined on $\Src \times \Src$.
These distances determine $(\Src, g)$ up to an isometry (see e.g. the proof of Proposition \ref{finding the metric tensor}  below).

Let $h \in C_0^\infty((0,\infty) \times \Src)$ and let us show that $w^h|_{(0,\infty) \times \Src}$ can be determined
from the inner products (\ref{inner_prod_w}). 
Let $\mathcal B$ be as in Lemma \ref{lem_nonrad_test}.
As $(0,\infty) \times \Src$ can be
covered with a countable number of sets of the form $\mathcal B$, it is enough to show that
$w^h|_{\mathcal B}$  can be determined.
We have already shown that $\mathcal N(\mathcal B)$
can be determined given (\ref{inner_prod_w}). 
Let $f\in \mathcal{N}(\mathcal B)$. Then $w^f$ is a solution of the following initial boundary value problem 
\begin{equation}
\label{initial/boundary value problem}
\begin{array}{l}
\p_t^2 w - \Delta_g w = f \quad \textrm{in } (0, \infty) \times \Src,
\\
w|_{\R \times \p\Src} = 0,
\\
w|_{t = 0} = \p_t w|_{t= 0} = 0. 
\end{array} 
\end{equation}
As $(\Src, g)$ is known, we can solve the above equation. Thus for every $f \in \mathcal{N}(\mathcal B)$ we are able to find $w^f$.
In particular, in the inner products
\begin{equation}
\label{cut w^f}
\langle w^f, \kappa^2 w^h \rangle_{L^2((0,\infty) \times \Src)}, \quad f \in \mathcal N(\mathcal B),
\end{equation}
the left factor $w^f$ is known. Observe that for any
$\phi \in C_0^\infty(\mathcal B)$
we have $w^f = \phi$ where $f = \Box_g \phi \in \mathcal N(\mathcal B)$, and therefore the inclusion 
\begin{align}
\label{nonrad_density}
\{w^f \;|\; f \in \mathcal{N}(\mathcal B) \} \subset L^2(\mathcal B)
\end{align}
is dense.
Hence we find $\kappa^2 w^h|_{\mathcal B}$ from the inner products \eqref{cut w^f}.

Let us conclude the proof by showing that function $\kappa|_{\Src}$ can be determined. 
We have already shown that,
when $f \in \mathcal N(\mathcal B)$,
both $w^f$ and $\kappa^2 w^h|_{\mathcal B}$
are determined by (\ref{inner_prod_w}).
Thus $\kappa|_{\Src}$ can be determined by the density (\ref{nonrad_density}).
\end{proof}

\section{The deterministic inverse problem}
\label{sec_deterministic}


In this section we prove Theorem \ref{th:main_geom}
in two steps: we show first that local the source-to-solution map $\Lambda_\Src$ determines a certain family of distance functions, and then that this family determines the geometry $g$.
We work first under the assumption that $d_g|_{\Src \times \Src}$ is known, and postpone the proof that $\Lambda_\Src$ determines 
$d_g|_{\Src \times \Src}$ in the end of the section.
Recall that in the previous section we already determined $d_g|_{\Src \times \Src}$, so the step from $\Lambda_\Src$ to 
$d_g|_{\Src \times \Src}$ is needed only in the proof of Theorem \ref{th:main_geom}.

\subsection{Reconstruction of a family of distance functions from the local source-to-solution mapping $\Lambda_{\Src}$}
Consider the following data:
\begin{equation}
\label{Response Data}
(\Src, g|_{\Src}, d_g|_{\Src \times \Src}, \; \Lambda_\Src)
\end{equation}
Here $\Src$ and $g|_{\Src}$ stand for the assumption that the Riemannian structure of the open manifold $\Src$ is known. We show the following theorem. 

\begin{theorem}
\label{th_geom} 
Let $(N,g)$ be a complete Riemannian manifold. 
Then the local source-to-solution data (\ref{Response Data})
determines the following family of distance functions
\begin{equation}
\label{distance functions}
R_\Src(N):=\{d_g(x,\cdot)|_{\Src}:x\in N\}\subset C(\Src).
\end{equation}

\end{theorem}
This is to be proved in several steps. 
Let $T, \epsilon >0$. For each $r>\epsilon$ and  $x \in N$ we define a set 
$$
S_\epsilon(x,r):=(T-(r-\epsilon),T)\times B(x,\epsilon)
$$

We denote for any measurable $A\subset N$ the function space 
$$
L^2(A):=\{u\in L^2(N): \hbox{supp}(U) \subset \overline{A}\}.
$$
Recall that for any $f\in C^\infty_0(\R_+\times N)$ the solution $w^f(T,\cdot)\in L^2(N)$.

\begin{lemma}
\label{crossing balls and functions}
Let $p,y,z\in N$, $\epsilon>0$ and $\ell_p,\ell_y,\ell_y >\epsilon$. Then the following are equivalent: 
\begin{enumerate}[(i)]
\item We have
\begin{equation}
\label{cond 3}
B(p,\ell_p)\subset \overline{B(y,\ell_y) \cup B(z,\ell_z)}.
\end{equation}
\item Suppose that
\begin{equation}
\label{cond 4}
\begin{array}{c}
\textrm{for all } f \in C^\infty_0(S_\epsilon(p,\ell_p)) \textrm{ there exists } (f_j)_{j=1}^\infty \subset C^\infty_0(S_\epsilon(y,\ell_y)\cup S_\epsilon(z,\ell_z)) 
\\
\textrm{ such that } \|w^f(T,\cdot)-w^{f_j}(T, \cdot)\|_{L^2(N)} \stackrel{j\rightarrow \infty}{\longrightarrow} 0.
\end{array}
\end{equation}
Here $w^f, w^{f_j}$ is the solution of \eqref{eq_wave_test2} with $\R^n$ replaced by $N$.
\end{enumerate}
\end{lemma}
\begin{proof}

Suppose that \eqref{cond 3} is valid. Let $ f \in C^\infty_0(S_\epsilon(p,\ell_p))$, then by the finite speed of wave propagation it holds that 
$$
\supp w^f(T) \subset B(p,\ell_p) \subset  \overline{B(y,\ell_y) \cup B(z,\ell_z)}
$$

Let $\chi(x)$ 
be the characteristic function of the ball $B(y,\ell_y)$ and set $w^f_y(T,x):=\chi(x) w^f(T,x)$ and $w^f_z(T,x):=w^f(T,x)-w^f_y(T,x)$. 
Since the boundary of a geodesic ball is a set of measure zero (see \cite{Oksanen1}), it holds that $w^f_y(T,\cdot)\in L^2(B(y,\epsilon))$ and $w^f_z(T,\cdot)\in L^2(B(z,\epsilon))$. 
By approximate controllability there exist sequences $(f_y^j)_{j=1}^\infty \subset C^\infty_0(S_\epsilon(y,\ell_y))$ and $(f_z^j)_{j=1}^\infty \subset C^\infty_0(S_\epsilon(z,\ell_z))$ such that sequences $(w^{f_y^j}(T,\cdot))_{j=1}^\infty$ and $(w^{f_z^j}(T,\cdot))_{j=1}^\infty$ converge to $w^f_y(T,\cdot)$ and $w^f_z(T,\cdot)$, respectively, in $L^2(N)$.
Therefore sequence 
$$
f_j=f^j_y+f_z^j \in C^\infty_0(S_\epsilon(y,\ell_y)\cup S_\epsilon(z,\ell_z)), \: j=1,2,\ldots$$ 
satisfies \eqref{cond 4}. 

Suppose that \eqref{cond 3} is not valid. Then the open set
$$
U:=B(p,\ell_p)\setminus \overline{(B(y,\ell_y)\cup B(z,\ell_z))}
$$
is not empty. By approximate controllabilty, we can choose $f \in  C^\infty_0(S_\epsilon(p,\ell_p))$ such that $\|w^f(T,\cdot)\|_{L^2(U)}>0$. By finite speed of wave propagation it holds that 
$$
\inf \{\|w^{f}(T,\cdot)-w^h(T,\cdot)\|_{L^2(N)}: h\in C^\infty_0(S_\epsilon(y,\ell_y)\cup S_\epsilon(z,\ell_z))\}>0.
$$
Therefore \eqref{cond 4} is not true.
%
\end{proof}
For any point $(p,\xi) \in TM, \|\xi\|_g=1$ we will denote the cut distance function
$$
\tau(p,\xi)=\sup\{t>0:d_g(p,\gamma_{p,\xi}(t))=t\}.
$$
Let $\alpha,\beta:(0,1) \rightarrow N$ be curves such that $\alpha(1)=\beta(0)$. Then we denote by $\alpha\beta$ the concatenated curve.
\begin{lemma}
\label{Finding cut dist}
Let $(N,g)$ be a complete Riemannian manifold. Let $x,y \in N$ and let $\gamma_{y,\xi}$ be a distance minimizing geodesic from $y$ to $x$. Let $s:=d_g(x,y)$. Let $r>0$. If $\tau(y,\xi)<s+r$, then 
\begin{eqnarray}
\label{Ball condition for cut distance}
&&\textrm{there exists }\epsilon >0 \textrm{ such that } B(x,r+\epsilon)\subset \overline{B(y,s+r)}.
\end{eqnarray}
Also if \eqref{Ball condition for cut distance} is valid then $\tau(y,\xi)\leq s+r.$

Moreover, we have  
$$
\tau(y,\xi)=\inf\{s+r>0: r,s >0, \gamma_{y,\xi}([0,s])\subset \Src, \eqref{Ball condition for cut distance} \hbox{ holds} \}.
$$
\end{lemma}
\begin{proof}
Let $r>0$ and denote $p=\gamma_{y,\xi}(s+r)$. 

\noindent
Suppose that \eqref{Ball condition for cut distance} is valid. Let $\delta\in (0,\epsilon)$ and consider a point 
$$
z=\gamma_{y,\xi}(s+r+\delta)\in B(x,r+\epsilon).
$$
By \eqref{Ball condition for cut distance} $d_g(z,y)\leq s+r$. Thus $\tau(y,\xi)< s+r+\delta$. Since $\delta$ was arbitrary we have $\tau(y,\xi)\leq s+r$.

Suppose that $\tau(y,\xi)< s+r$. We show first that
\begin{equation}
\label{eq: closure inclusion}
\overline{B(x,y)}\subset B(y,s+r).
\end{equation}
By triangle inequality it suffices to show that  $\p B(x,y)\subset B(y,s+r)$. Let $z \in \p B(x,r)$. By triangle inequality $d_g(z,y)\leq s+r$. Let $\alpha$ be a minimizing geodesic from $x$ to $z$. Suppose first that $\alpha$ is not the geodesic continuation of segment $\gamma_{y,\xi}([0,s])$. Since a curve $\gamma_{y,\xi}\alpha$ has a length $s+r$ and it is not smooth at $x$, it must hold that $d_g(z,y)<s+r$. If $\alpha$ is the geodesic continuation of segment $\gamma_{y,\xi}((0,s))$, then $z=\gamma_{y,\xi}(s+r)=p$. Since $\tau(y,\xi)< s+r$, it holds that $d_g(y,p)<s+r$. Thus \eqref{eq: closure inclusion} follows. Therefore dist$_g(\p B(x,r), \p B(y,s+r))>0$ and \eqref{Ball condition for cut distance} is valid.

\end{proof}

Next we provide a method to find the cut distance function $\tau$.
\begin{proposition}
\label{reconstruction of cut distance from data}
For any $y\in \Src$ and $\xi\in S_yN$ we can find $\tau(y,\xi)$ from the local source-to-solution data \eqref{Response Data}. 
\end{proposition}
\begin{proof}
Let $y\in \Src$ and  $\xi\in S_yN$. Given the data \eqref{Response Data} we can find the geodesic segment $\gamma_{y,\xi}([0,s])$ for small values $s>0$. 

Let $s>0$ be so small that $\gamma_{y,\xi}([0,s])\subset \Src$. We denote $x=\gamma_{y,\xi}(s)$. Let $r>0$. Consider the relation \eqref{Ball condition for cut distance}. By Lemma \ref{Finding cut dist}, relation \eqref{Ball condition for cut distance} determines $\tau(y,\xi)$.

Choose $\epsilon>0$ so small that
$$
B(y,\epsilon)\cup B(x,\epsilon)\subset \Src.
$$
By taking $z=y$, $\ell_y=r+s=\ell_z$, $\ell_x=r+\epsilon$ as in Lemma \ref{crossing balls and functions} we see that \eqref{Ball condition for cut distance} is equivalent with relation \eqref{cond 4}. Using the Blagovestchenskii identity (see \eqref{Blagovestchenskii identity} in the appendix) we see that the source-to-solution data \eqref{Response Data} determines \eqref{cond 4}.
\end{proof}


\begin{lemma}
\label{minimizing geodesics of S fill N}
It holds that
$$\{\gamma_{y,\xi}(t) \in N: \: y\in \Src, \; \xi \in S_yN, \; t< \tau(y,\xi)\}=N.$$
\end{lemma}
\begin{proof}
Let $p\in N$ and choose any $y \in \Src$. Let $\gamma_{y,\xi}$ be a distance minimizing geodesic from $y$ to $p$. We denote by $r=d_g(y,p)$. Then it holds that $r\leq \tau(y,\xi)$. Choose $s\in (0,r)$ such that $y_1:=\gamma_{y,\xi}(s),\: \gamma_{y,\xi}([0,s]) \subset \Src$. Let $\xi_1:=\dot{\gamma}_{y,\xi}(s)$. We will show that $r-s< \tau(y_1,\xi_1)$ and this proves the claim of this lemma. 

Suppose that $\tau(y_1,\xi_1)\leq r-s$. By the symmetry of cut points, it holds that $\tau(p,\eta)\leq r-s$, where $\eta:=-\dot{\gamma}_{y,\xi}(r)$. Thus there exists $t \in(0,s)$ such that for a point $z=:\gamma_{y,\xi}(t)$ it holds $d_g(p,z)<r-t$. Then it  also holds that
$$
d_g(y,p)\leq d_g(y,z)+d_g(z,p)<t+r-t=r.
$$
This is a contradiction and therefore $r-s<\tau(y_1,\xi_1)$.
\end{proof}
Notice that the assumption $\Src$ is open is crucial in Lemma \ref{minimizing geodesics of S fill N}. For instance consider the cylinder 
$$
\{e^{i\pi t}\in \C:t\in [-1,1]\}\times (-1,1),
$$ 
and let $\Src = \{1\}\times(-1,1)$ and $p=(-1,0)$. Then it holds that every point in $\Src$ is a cut point of $p$.
\begin{proposition}
\label{finding the distance function}
Let $z, y \in \Src$, $\xi \in T_y\Src$, $\|\eta\|=1$ and $\widetilde r < \tau(y, \eta)$. Then the local source-to-solution data \eqref{Response Data} determines $d_g(p,z)$, where $p = \gamma_{y, \xi}(\widetilde r)$. 
\end{proposition}
\begin{proof}
Let $s\in (0,\widetilde{r})$ be such that $\gamma_{y,\xi}([0,s])\subset \Src$. We denote by $x=\gamma_{y,\xi}(s)$. Let $r:=\widetilde{r}-s$. 
 
Let $R>0$. By Lemma \ref{crossing balls and functions} the inclusion 
\begin{equation}
\label{crossing balls with closure}
B(x,r+\epsilon)\subset \overline{B(y,r+s) \cup B(z,R)}
\end{equation}
is valid  for all  $\epsilon>0$ small enough if and only if the equation \eqref{cond 4} is valid with $\ell_x=r+\epsilon$, $\ell_y=r+s$ and $\ell_z=R$. Using the Blagovestchenskii identity the local source-to-solution data \eqref{Response Data} determines  \eqref{cond 4}.
We will show that
$$
d_g(p,z)=R^\ast:=\inf\{R>0: \textrm{Formula }\eqref{crossing balls with closure} \textrm{ is valid for } R \hbox{ and some } \epsilon>0\}.
$$

Suppose that \eqref{crossing balls with closure} is valid. Since we assumed that $r+s<\tau(y,\xi)$, it holds that $p \in \overline{B(z,R)}$. Thus $d_g(p,z)\leq R^\ast$. 

Suppose that $R \in (d_g(p,z), R^\ast)$. Then for any $\epsilon>0$ \eqref{crossing balls with closure} is not valid. Choose for every $k\in \N$  a point  
$$
p_k \in B(x,r+1/k) \setminus \overline{B(y,r+s) \cup B(z,R)}.
$$
By compactness of $\overline{B(x,r+1)}$ we may assume that $p_k \rightarrow \widetilde p \in \p B(x,r)$ as $k\rightarrow \infty$. By similar argument as in the proof of Lemma \ref{Finding cut dist} we deduce that $\widetilde p=p$. Since $p \in B(z,R)$ we get a contradiction with the choice of sequence $(p_k)_{k=1}^\infty$. Therefore interval $(d_g(p,z),  R^\ast)=\emptyset$ and $R^\ast =d_g(p,z)$.
\end{proof}
Let $p \in N$ and $z \in \Src$. By Lemma \ref{minimizing geodesics of S fill N} it holds that there exists $y\in \Src$ and an unit vector $\xi\in S_yN$ such that $p=\gamma_{y,\xi}(\widetilde{r})$, for some $\widetilde{r}< \tau(y,\xi)$. By Propositions \ref{reconstruction of cut distance from data} and  \ref{finding the distance function} we have reconstructed $R_{\Src}(N).$ Therefore Theorem \ref{th_geom} is proved.

\subsection{Reconstruction of the Riemannian manifold from the distance functions}

\label{Reconstruction of the Riemannian manifold from the distance functions}
So far we have been able to find the following \textit{distance data}
\begin{equation}
\label{distance data}
(\Src, g|_{\Src}, R_\Src(N)),
\end{equation}
where $R_\Src(N)$ is defined by \eqref{distance functions}. In this section we will show, how one can reconstruct the topological, smooth and Riemannian structures from the  distance data \eqref{distance data}.  The rest of the paper is devoted to showing the following theorem:

\begin{theorem}
\label{th:distance functions define manifold}
Let $(N,g)$ be a complete smooth Riemannian manifold without a boundary. Let $U \subset N$ be open, bounded and have a smooth boundary. Suppose that the topological and smooth structure of $U$ are known, and  $g|_{U}$ is also known. Then 
$$
R(N):=\{d_g(\cdot,x)|_{\overline{U}}: x \in N\}\subset C(\overline{U})
$$
determines, topological, smooth and Riemannian structure of $N$ up to isometry.
\end{theorem}

We emphasize that in \cite{Katchalov2001, kurylev1997multidimensional}  similar results and methods of the proofs have been considered in the case of manifold with a boundary. 

Since $\overline{U}$ is compact, $C(\overline{U})$ is a Banach space when equipped with $L^\infty$-norm. We define the mapping 
$$
R:N \rightarrow C(\overline{U}), \: R(x)=r_x=d_g(x,\cdot)|_{\overline{U}}.
$$
Our aim is to construct such a Riemannian structure in $R(N)\subset C(\overline{U})$ that $R:N\rightarrow R(N)$ is a Riemannian isometry.
\begin{lemma}
Mapping $R$ is continuous and one-to-one.
\end{lemma}
\begin{proof}
Let $x,y \in N$. Then by the triangle inequality
$$
\|R(x)-R(y)\|_{L^\infty(\overline{U})}=\sup_{z\in \overline{U}} |r_x(z)-r_y(z)| \leq d_g(x,y).
$$
Thus $R$ is continuous.

Suppose that $x,y\in N$ satisfy $r_x=r_y$. If $x \in \overline{U}$ then $r_y(x)=0$ and thus $x=y$. Therefore we can assume that $x,y \in N \setminus \overline{U}$. Since $\overline{U}$ is compact there exists a closest point $z\in \overline{U}$ to $x$. Then $z \in \p U$ and it is also a closest point of $\overline{U}$ to $y$. Since $\p U$ is smooth $n-1$ dimensional submanifold of $N$, the distance minimizing unit speed geodesic $\gamma$ from $z$ to $x$ is orthogonal to $\p U$. Since both $x$ and $y$ are points of the exterior of $U$, it  holds by the uniqueness of geodesics that 
$$
x=\gamma(r_x(z))=\gamma(r_y(z))=y.
$$
This completes the proof.
\end{proof}

Next we will recall two topological results that allow us to prove that mapping $R:N \rightarrow R(N)$ is a homeomorphism. 

\begin{definition}
Let $X$ be a topological space. We say that a sequence $(x_j)_{j=1}^{\infty}$ in $X$  escapes to infinity, if for every compact $K \subset X$, $x_j \in K$ for at most finitely many $j \in \N$. 
\end{definition}
For the proofs of the following two lemmas see for instance \cite{Willard}.
\begin{lemma}
Let $(X,d_X)$ and $(Y,d_Y)$ be metric spaces. Let $f:X\rightarrow Y$ be continuous. Then $f$ is proper if and only if for every sequence $(x_j)_{j=1}^\infty \subset X$ that escapes to infinity the image sequence $(f(x_j))_{j=1}^\infty \subset Y$ escapes to infinity.
\end{lemma}

%

\begin{lemma}
\label{proper maps are closed}
Let $(X,d_X)$ and $(Y,d_Y)$ be metric spaces. Let $f:X\rightarrow Y$ be one-to-one, continuous and proper. Then mapping the $f$ is closed.
\end{lemma}

\begin{proposition}
\label{Mapping R is homeo}
Mapping $R:N \rightarrow R(N)$ is a homeomorphism. 
\end{proposition}
\begin{proof}
If $N$ is bounded, then $N$ is compact, and the claim follows from basic topology. Suppose that $N$ is not bounded. Let $(x_j)_{j=1}^\infty \subset N$ be a sequence that escapes to infinity. Let $x_0 \in \overline{U}$. We define $X_j:=\overline{B(x_0,j)}$ for every $j\in \N$ and $Y_j=R(X_j)$. Then $\cup_{j=1}^\infty X_j=N$ and thus 
$$
\lim_{j\rightarrow \infty }d_g(x_0,x_j)=\infty.
$$ 
We write $R(x_0)=:r_0$ and $R(x_j)=:r_j$. Then 
$$
d_\infty(r_{0},r_j)\geq |d_g(x_0,x_0)-d_g(x_0,x_j)|=d_g(x_0,x_j).
$$
Thus $d_\infty(r_{0},r_j) \longrightarrow \infty $ as $j\longrightarrow \infty$. Since a compact set of a metric space is always bounded, it holds that sequence $(r_j)_{j=1}^\infty$ escapes to infinity. Therefore $R$ is a proper mapping and by Lemma \ref{proper maps are closed} it is closed.
\end{proof}
By Proposition \ref{Mapping R is homeo}, the topological structure of $N$ has been found. Next we will show, how to construct such a smooth atlases on $N$ and $R(N)$ that the mapping $R$ is a diffeomorphism.

Let $z\in \overline{U}$ and $x\in N$. Denote by $\omega(x)$ the cut locus of $x$.  Recall that $r_x:=d_g(x,\cdot)|_{U}$ is smooth at $z$ if and only if $z\neq x $ or $z\notin \omega(x)$ (see Lemma 2.1.11 and Theorem  2.1.14 of \cite{klingenberg}). Using also the fact that $z\in \omega(x)$ if and only if $x\in \omega(z)$ we can find the cut locus $\omega(z)$ from data \eqref{distance data}. We write 
$$
I(z)\subset T_zN,
$$ 
for the largest, open star like subset of $T_zN$ such that the exponential mapping $\exp_z:T_zN\to N$   
restricted to $I(z)$ is a diffeomorphism onto an open set 
$$
\exp_z(I(z))=N \setminus \omega(z).
$$

We define a mapping $\Phi_{z}$ by 
$$
\Phi_z(r) := -r(z) \nabla_g r|_z \in I(z), \: r\in R(\exp_z(I(z))).
$$ 
By the following lemma it holds 
\begin{equation}
\label{eq: coordinate mapping}
\Phi_z \circ R|_{ R(\exp_z(I(z)))} = \exp_z^{-1},
\end{equation}
\begin{lemma}
\label{lem:lemma}
Let $x\in N$. Then the following are equivalent: 
\begin{eqnarray}
&& \eta \in I(z) \label{x in geodesic ball} \hbox{ and } \exp_z(\eta)=x
\\
&&\nabla_g d_g(x,\cdot)|_{z}\in T_{z}N \hbox{ exists and } \eta=-d_g(x,z)\nabla_g d_g(x,\cdot)|_{z}. \label{gradient test for points in geodesic ball}
\end{eqnarray}
\end{lemma}
\begin{proof}
Suppose that formula \eqref{x in geodesic ball} is valid. Since exponential mapping $\exp_{z}|_{I(z)}$ is a diffeomorphism, the point $z$ is not in the cut locus of $x$ and therefore the function $d_g(x,\cdot)$ is smooth at $z$. Thus $\nabla_g d_g(x,\cdot)|_{z}\in T_{z}N$ exists and $\eta=-d_g(x,z)\nabla_g d_g(x,\cdot)|_{z}$. Therefore \eqref{gradient test for points in geodesic ball} is also valid.

Suppose that formula \eqref{gradient test for points in geodesic ball} is valid. Then it holds that $d_g(x,\cdot)$ is smooth at $z$. Thus $x$ is not in the cut locus of $z$ and therefore $\xi:=-\nabla_g d_g(x,\cdot)|_{z}$ is the initial velocity of the unique distance minimizing geodesic from $z$ to $x$. We have
$$
\exp_{z}(\eta)=\gamma_{z, \xi}( d_g(x,z))=x\in \exp_{z}(I(z)).
$$
\end{proof}
We define the smooth structure on $R(N)$ by using mappings $\Phi_z, z \in \overline{U}$. By Lemma \ref{minimizing geodesics of S fill N}  we have $\cup_{z\in \overline{U}}\hbox{dom}(\Phi_z)=N,$ and by \eqref{eq: coordinate mapping} each mapping $\Phi_z$ is a topological coordinate mapping. Let $z, w \in \overline{U}$. Then the composition 
$$
\Phi_{z} \circ \Phi_{w}^{-1}=(\Phi_z\circ R)\circ (\Phi_w\circ R)^{-1}=\exp_z^{-1}\circ \exp_w
$$ 
is well defined and smooth in the set 
$$
I(w)\bigcap (\exp_w^{-1}\circ\exp_z)(I(z)) \subset T_wN.
$$ 
Moreover, $R$ is clearly smooth when the smooth structure of $R(N)$ is defined in this way. Therefore we have proved the following proposition.
\begin{proposition}
\label{Mapping R is diffeo}
The mapping $R:N \rightarrow R(N)$ is a diffeomorphism.
\end{proposition}

We define a metric tensor $\widetilde g:=(R^{-1})^\ast g$ on $R(N)$, that is, $\widetilde{g}$ is the push forward of $g$. Then $(R(N),\widetilde g)$ and $(N,g)$ are Riemannian isometric. In the next proposition, we provide a method to construct  representation of $\widetilde g$ in local coordinates of $R(N)$.
 
\begin{proposition}
\label{finding the metric tensor}
Let $\widetilde{g}:=(R^{-1})^\ast g$. We can construct the metric tensor $\widetilde{g}$ on $R(N)$ from the distance data \eqref{distance data}.
\end{proposition}
\begin{proof} Let $r_0\in R(N)$. We write $x_0:=R^{-1}(r_0)$. By Lemma \ref{minimizing geodesics of S fill N} it holds that there exists a point $z\in U$ that is not in the cut locus of $x_0$. Let $U'\subset U$ be an open neighborhood of $z$ such that $d_g(\cdot,y)$ is smooth at $x_0$ for any $y\in U'$. 

It holds that
$$
\nabla_g d_g(\cdot,y)|_{x_0}=-\dot{\gamma}_{y,x_0}(d_g(y,x_0))\in S_{x_0}N,
$$
where $\gamma_{y,x_0}$ is the unique unit speed distance minimizing geodesic from $y$ to $x_0$ 
Since $U'$ is open and $\exp_{x_0}$ is continuous the set $\exp_{x_0}^{-1}U'\subset T_{x_0}N$ is open. Therefore the set
$$
\mathcal V:=\{\nabla_g d_g(\cdot,y)|_{x_0}\in S_{x_0}N: y\in U'\}
$$ 
is open in $S_{x_0}N$. Let $(x,\xi)\in TN$. We will use the notation
$$
\xi^\flat:=\langle \xi,\cdot\rangle_g\in T^\ast_xN.
$$

Since $R$ is a diffeomorphism it holds that
$$
\mathcal{W}^\ast:=R_{\ast}\mathcal{V}^{\ast}=\{(\nabla d_g(R^{-1}(\cdot),y)|_{r_0})^\flat\in S^{\ast}_{r_0}R(N): y\in U'\} 
$$
is open. For any point $y\in U'$ we define an evaluation function $E_y: R(N) \rightarrow \R$ with the formula $E_y(r)=r(y).$ Notice that
$$
d E_y|_{r_0}=(\nabla d_g(R^{-1}(\cdot),y)|_{r_0})^\flat,
$$
and therefore
$$
\mathcal{W}^\ast=\{d E_y|_{r_0}\in S^\ast_{r_0}R(N): y\in U'\}.
$$

As we know the smooth structure of $R(N)$ we can find the set $\mathcal{W}^\ast$. The last step is to show that set $\mathcal{W}^\ast$ determines $\widetilde g(r_0)$.

Let 
$$
\R_+\mathcal{W}^\ast:=\{sv \in T^\ast_{r_0}R(N): v\in \mathcal{W}^{\ast},s>0\}
$$
be the open cone generated by $\mathcal{W}^\ast$. Let $\{E_j\}_{j=1}^{n}$ be a local coordinate system at $r_0$. For any $s>0$ and $v\in \mathcal{W}^\ast$ it holds in coordinates $\{E_j\}_{j=1}^{n}$ that
$$
F(sv):=s^2\widetilde g^{ij}(r_0)v_{i}v_j=s^2.
$$
We know the function $F:\R_+\mathcal{W}^\ast\rightarrow \R$, and $\R_+\mathcal{W}^\ast$ is open, we get
$$
\widetilde g^{ij}(r_0)=\frac{\partial}{\partial E_i}\frac{\partial}{\partial E_j}F.
$$
\end{proof}
By Propositions \ref{Mapping R is homeo}, \ref{Mapping R is diffeo} and \ref{finding the metric tensor} we can reconstruct $(R(N),\widetilde g)$, more over $(N,g)$ and $(R(N),\widetilde g)$ are isometric as Riemannian manifolds. Thus  we have proved Theorem \ref{th:distance functions define manifold}. 


In order to prove Theorem \ref{th:main_geom} we still need the next small lemma.

\begin{lemma}
\label{distance from responce mapping}
Let $(N,g)$ and $\Src$ be as in the formulation of Theorem \ref{th:main_geom}. Then data $(\Src, \Lambda_\Src) $ determines the distance function $d_g$ on $\Src \times \Src$.
\end{lemma}
\begin{proof}
Let $x,y\in \Src$. Since $\Src$ is a smooth manifold, we may choose an auxiliary metric $d_0$ on $\Src$ that gives the same topology as $g$. Let $\epsilon >0$ and consider the metric ball $B_{d_0}(x,\epsilon).$
We write $\mathcal{B}_\epsilon:=(0,\infty)\times B_{d_0}(x,\epsilon)$ and
$$
t_\epsilon=\inf \{t >0: \hbox{there is } f\in C_0^\infty(\mathcal{B}_\epsilon) \hbox{ such that supp}(\Lambda_\Src f)(t,\cdot)\cap B_{d_0}(y,\epsilon)\neq \emptyset\}. 
$$
By the finite speed of wave propagation  and the  approximate controllability the equality
$$t_\epsilon = \hbox{dist}_g(B_{d_0}(x,\epsilon),B_{d_0}(y,\epsilon))$$
holds. Thus the following limit is valid
$$
d_g(x,y)=\lim_{\epsilon \rightarrow 0}t_\epsilon.
$$
\end{proof}
Now we are finally ready to give a proof for Theorem \ref{th:main_geom}.
\begin{proof}[Proof of Theorem \ref{th:main_geom}]

By making $\Src_i$ smaller, if needed, we may assume without loss of generality that $\Src_i$ is precompact with smooth boundary and that $\phi:\overline{\Src}_1\rightarrow \overline{\Src}_2$ (see \eqref{diffeo phi}) is a diffeomorphism. Denote $ R(N_i)=\{d_{i}(x, \cdot)|_{\overline{\Src}_i}:x\in N_i\}$ and consider a mapping 
$$
R_i:N_i\rightarrow R(N_i), \: i=1,2, \: R_i(x)=d_i(x,\cdot)|_{\overline{\Src_i}}.
$$
By Lemma \ref{distance from responce mapping}, Proposition \ref{finding the metric tensor} and equation \eqref{eq: phi pulls back Lambda} it holds that 
\begin{equation}
\label{eq:distance_&_metric_in_Src_coincide}
d_1(\cdot, \cdot)|_{\Src_1\times \Src_1}=d_2(\phi(\cdot), \phi(\cdot))|_{\Src_1\times \Src_1} \hbox{ and } g_1|_{\Src_1}=\phi^{\ast}g_2|_{\Src_2}.
\end{equation}
Therefore we may assume that $\phi:\overline{\Src}_1\rightarrow \overline{\Src}_2$ is a Riemannian isometry. By Proposition \ref{Finding cut dist} the following relation
$$
\tau_1(y,\xi)=\tau_2(\phi(y),\phi(\xi)), \: y\in \Src_1, \: \xi\in S_yN_1
$$
is valid. Therefore by Proposition \ref{finding the distance function} it holds that
$$
R_2(N_2)=\Phi(R_1(N_1)),
$$
where 
$$
\Phi:C(\overline{\Src}_1) \rightarrow C(\overline{\Src}_2), \:  \Phi(f)=f\circ \phi^{-1}.
$$
Moreover by Theorems \ref{th_geom} and \ref{th:distance functions define manifold} the mappings $R_i:N_i\rightarrow R_i(N_i)$ are Riemannian isometries. 
%

With out loss of generality we assume that $\overline{\Src_1}\subset V$, where $(V,\alpha)$ is a coordinate chart for $N_1$. Write $\alpha \circ \phi=:\widetilde \alpha$, $W=\alpha(V)$ and define Riemannian isometries
$$
\alpha^\ast:R(N_1)\rightarrow \alpha^\ast(R_1(N_1)) \subset  C(W), \: \alpha^\ast(r)(x)=r(\alpha^{-1}(x))
$$ 
and 
$$
\widetilde \alpha^\ast:R(N_2)\rightarrow \widetilde \alpha^\ast(R_2(N_2)) \subset C(W), \: \widetilde \alpha^\ast(r)(y)=r(\widetilde \alpha^{-1}(y)).
$$

Thus we have proved that mapping 
$$
N_1\stackrel{R_1}{\longrightarrow} R_1(N_1)\stackrel{\alpha^\ast}{\longrightarrow} \alpha^\ast(R_1(N_1))   \stackrel{id}{\longrightarrow} \widetilde \alpha^\ast(R_2(N_2))\stackrel{(\widetilde \alpha^\ast)^{-1}}{\longrightarrow} R_2(N_2)\stackrel{R_2^{-1}}{\longrightarrow} N_2,
$$
is a Riemannian isometry. This ends the proof.


\end{proof}

Lastly we will give a proof for Corollary \ref{co:spectral problem}.

\begin{proof}[Proof of Corollary \ref{co:spectral problem}]
Since $N$ is a compact manifold without a boundary we have
$$
0=\lambda_1 < \lambda_2 \leq \lambda_3 \ldots \: .
$$
Let $f \in C^\infty_0((0,\infty) \times \Src)$ and $w=w^f$ be the solution of the initial value problem \eqref{eq:problem_RM}. For each $j\in \N$ we define the $j^{th}$ Fourier coefficient 
$$
I_j(t)=\langle w(t,\cdot), \varphi_j \rangle_{L^2(N)}.
$$
Since $w$ is smooth, also $I_j$ is smooth. By Greens formula and the initial conditions of \eqref{eq:problem_RM} it holds that
\begin{equation}
\label{eq:spectral ODY}
\left\lbrace \begin{array}{l}
\frac{d^2}{dt^2}I_j(t)-\lambda_jI_j(t)=\int_\Src f(t,x)\varphi_j(x) dV_g(x)
\\
I_j(0)=\frac{d}{dt}I_j(0)=0.
\end{array}\right.
\end{equation}

Solve the ordinary differential equation \eqref{eq:spectral ODY} to get 
$$
 I_j(t)=\int_0^t\int_\Src s_j(t-s)f(s,x)\varphi_j(x) dV_g(x) ds, \: j\geq 1
$$
where
$$
s_1(t)=t \hbox{ and } s_j(t)=\frac{\sin(\sqrt{\lambda_j}(t))}{\sqrt{\lambda_j}}, \hbox{ for  } j>1.
$$
Notice that apriori the volume form $dV_g|_{\Src}$ is not given. However without a loss of generality we may assume that $\Src$ is contained in a coordinate patch of $N$. Thus we can assume, that we are given some volume form $\omega$ on $\Src$. Therefore there exists a unique smooth function $\eta:\Src\rightarrow (0,\infty)$ such that 
$$
\eta dV_g|_{\Src}= \omega.
$$
We write
$$
\widetilde{I}_j(t)=\int_0^t\int_\Src s_j(t-s)f(t,x)\varphi_j(x) \omega(x) ds.
$$
By direct computations and initial values of \eqref{eq:spectral ODY} we have
\begin{equation}
\label{eq:waave with source eta f}
\sum_{j=1}^\infty \widetilde I_j(t)\varphi_j(x)=w^{\eta f}(t,x).
\end{equation}

Thus for every $f\in C^\infty_0((0,\infty) \times \Src)$ the Fourier coefficients $\widetilde I_j(t)$ can be recovered from the Spectral data \eqref{Spectral data}. We conclude that we have recovered the mapping
$$
\sum_{j=1}^\infty \widetilde I_j(t)\varphi_j(x)|_{{x\in \Src}}=w^{\eta f}(t,x)|_{x\in \Src}=(\Lambda_\Src M_\eta)f.
$$
Here $M_\eta $ is the multiplier operator $M_\eta f(t,x)=\eta(x) f(t,x)$.  Let $R(h(t,x))=h(-t,x)$. Then
$$
(\Lambda_\Src M_\eta)^\ast=M_\eta^\ast\Lambda_\Src^\ast=M_\eta R\Lambda_\Src R=R(M_\eta \Lambda_\Src)R
$$
(see Lemma \ref{Adjoint of solution mapping} in the appendix), so that we have recovered the operator $M_\eta \Lambda_{\Src}$. Notice that the unknown weight $\eta$ can be found in the same way as the function $\kappa$ in the proof of Theorem \ref{th_response}. Therefore the claim follows from Theorem \ref{th:main_geom}.

\end{proof}

\noindent
\textbf{Acknowledgements.} The research of TH, ML, and TS was partly supported by  the Finnish Centre of Excellence in Inverse Problems Research
and Academy of Finland. In particular, ML was supported by projects 284715 and 303754, TH by project 275177 and TS by projects 273979 and 263235. LO was partly supported by EPSRC grant EP/L026473/1

The authors thank M. Santacesaria for the help in preparing the article.
\section{Appendix}

In the appendix we recall some well known results related to the propagation of waves on Riemannian manifold. We will use the assumptions and notations of Theorem \ref{th:main_geom}. Let $T>0$,
$p \in N$ and $a>1$. Let $C_{p,T}$ be the cone 
$$
C_{p,T}:=\{(t,q)\in \R\times N: 0\leq t \leq T, \: d_{N}(p,q)< T- t\}.
$$ 

\begin{theorem}[Finite speed of propagation]
\label{[Finite speed of propagation cone version} Let $f\in L^2( \R\times N)$. Suppose that $u$ solves 
$$
\left\lbrace \begin{array}{l}
(\p^2_t-\Delta_g)u=f, \: \textrm{ in } (0,\infty)\times N
\\
 f|_{C_{p,T}}=0 
\\
u|_{B(p,T)\times \{t=0\}}=\p_tu|_{B(p,T)\times \{t=0\}}=0, 
\end{array} \right.
$$
Then 
$$
u|_{C_{p,T}}=0. 
$$

\end{theorem}
\begin{proof}
See \cite{Taylor}.
\end{proof}

Consider an open double cone created by a cylindrical set $(0,2T)\times \Src$ 
$$
C(T,\Src)=\lbrace(t,x)\in (0,2T)\times N: 
\hbox{dist}_g(x,\Src)< \min\{t,2T-t\}
\rbrace
$$ 
We write 
$$
M(T,\Src)=\{x \in N: \hbox{dist}_g(x,\Src)\leq T\},
$$
for the domain of influence of set $\Src$. 
\begin{theorem}[Tataru's unique continuation]
\label{Unique continuation}
Let $\Src\subset N$ be open and bounded. Let $u\in C^\infty_0(\R\times N)$. Suppose that $(\p^2_t-\Delta_g)u=0$ in $(0,2T)\times M(T,\Src)$ and $u|_{(0,2T)\times \Src}\equiv 0$. Then $u|_{C(T,\Src)}\equiv 0$. 
\end{theorem}
\begin{proof}
See \cite{Katchalov2001} for a local result and \cite{Tataru} for the global result.

\end{proof}
%
We use a short hand notation
$$
\mathcal{F}_{\Src,T}:=\{f\in C^{\infty}_0(\R \times N): \supp f \subset (0,T) \times \Src\}.
$$
The Tataru's unique continuation result yields immediately the following controllability results.
\begin{theorem}[Approximate controllability] 
\label{density}
Let $\Src \subset N$ be open and bounded. For any $T>0$ set 
$$
\mathcal{W}_T:=\{w^f(T): f\in \mathcal{F}_{\Src,T}\}
$$
is dense in Hilbert space $L^2(M(T,\Src))$.
\end{theorem}
\begin{proof}
By the finite speed of wave propagation $\mathcal{W}_T \subset L^2(M(T,\Src))$. Since $L^2(M(T,\Src))$ is a Hilbert space, it suffices to prove that $\mathcal{W}_T^{\perp}=\{0\}$. Suppose that $\phi \in L^2(M(T,\Src))$ is such that $(w^f(T),\phi)_{L^2(N)}=0$ for all $f \in \mathcal{F}_{\Src,T}$. Let $u \in C^{\infty}(\R \times N)$ solve
\begin{equation}
\label{wave equ 2}
\left\lbrace \begin{array}{l}
(\p^2_t-\Delta_g)u=0, \: \textrm{ in } (0,T) \times N
\\
u|_{t=T}=0, \;\p_tu|_{t=T}=\phi.
\end{array} \right.
\end{equation}
Let $f \in \mathcal{F}_{B,T}$. By the finite speed of wave propagation, there exists a compact set of $N$ that contains the $\supp w^f(t)$ for each $t \in (0,T)$.  We use the Green identities to see that
$$
\langle f,u\rangle _{L^2((0,T) \times N)}=\langle \square_g w^f,u\rangle _{L^2((0,T) \times N)}-\langle w^f,\square_g u\rangle _{L^2((0,T) \times N)}=0.
$$
Since $\mathcal{F}_{\Src,T}$ is dense in $L^2((0,T)\times \Src)$, it holds that $u \equiv 0$ in $(0,T]\times \Src$. 

Let $U$ solve
\begin{equation}
\label{wave equ 3}
\left\lbrace \begin{array}{l}
(\p^2_t-\Delta_g)U=0, \: \textrm{ in } (0,2T) \times N
\\
U|_{t=0}=u(0), \;\p_tU|_{t=0}=\p_t u|_{t=0}.
\end{array} \right.
\end{equation}
By equations \eqref{wave equ 2} and \eqref{wave equ 3} it holds $U|_{[0,T]\times N}=u$. More over the function $\widetilde{u}(t,x)=-u(2T-t,x)$ solves the wave equation 
\begin{equation}
\label{wave equ 4}
\left\lbrace \begin{array}{l}
(\p^2_t-\Delta_g)\widetilde{u}=0, \: \textrm{ in } (T,2T) \times N
\\
\widetilde{u}|_{t=T}=0, \;\p_t\widetilde{u}|_{t=T}=\phi,
\end{array} \right.
\end{equation}
since $\widetilde{u}(T,x)=-u(2T-T,x)=0 \textrm{ and } \p_t\widetilde{u}|_{t=T}=\p_tu(2T-T)=\phi$.
Therefore in particular $U|_{(0,2T)\times \Src}\equiv 0$.

By unique continuation (Theorem \ref{Unique continuation}), it holds that $U|_{C(T,\Src)}\equiv 0$. Since $M(T,\Src)\times \{T\}\subset C(T,\Src)$ we have 
$$\phi|_{M(T,\Src)}=\p_t U|_{t=T}|_{M(T,\Src)} =0.$$ 
\end{proof}

Next our aim is to prove the Blagovestchenskii identity on a complete Riemannian manifold $(N,g)$. This identity was originally introduced in \cite{blagoveshchenskii1971inverse, blagovestchenskii1969one} for a Riemannian manifold with boundary.
\begin{theorem}
\label{th_Blago}
Let $(N,g)$ be a complete Riemannian manifold. Let $T>0$, $\Src \subset N$ be open and bounded.  
Let $f,h \in \mathcal{F}_{\Src,2T}$, then 
\begin{equation}
\label{Blagovestchenskii identity}
\langle w^f(T, \cdot),w^h(T, \cdot)\rangle_{L^2(N)}=\langle f,(J\Lambda_{\Src} -\Lambda_{\Src}^\ast J)h\rangle_{L^2((0,T)\times N)}
\end{equation}
where the operator $J:L^2(0,2T)\rightarrow L^2(0,T)$ is defined as
$$
J\phi(t)=\frac{1}{2}\ \int_{t}^{2T-t}\phi(s) \;ds.
$$
\end{theorem}
\begin{proof}

Let $f,h \in \mathcal{F}_{B,2T}$ and consider the mapping $W:[0,2T]\times [0,2T]\rightarrow \R$,
$$
W(t,s)=\langle w^f(t),w^h(s)\rangle_{L^2(N)}.
$$
Then using Greens formula 
$$
(\p^2_t-\p^2_s)W(t,s)=(\p^2_t-\p^2_s)\langle w^f(t),w^h(s)\rangle_{L^2(N)}
$$ 
$$
=\langle f(t),\Lambda_{B,2T}h(s)\rangle_{L^2(N)}-\langle\Lambda_{B,2T}f(t),h(s)\rangle_{L^2(N)}:=F(t,s).
$$
Notice that there is no boundary terms due finite speed of wave propagation. The function  $(t,s)\mapsto F(t,s)$ can be computed, if the local source-to-solution mapping $\Lambda_{\Src}$ is given. By \eqref{eq:problem_RM} it holds that
$$
W(0,s)=0=\p_t W(t,s)|_{t=0}.
$$
Thus $w$ is the solution of the following $(1+1)$-dimensional initial value problem:
\begin{equation}
\label{1D wave equ}
\left\lbrace \begin{array}{l}
(\p^2_t-\p_s^2)W=F, \: \textrm{ in } (0,2T)\times \R
\\
W|_{t=0}=\p_tW|_{t=0}=0.
\end{array} \right.
\end{equation}
Recall that the following formula
\begin{equation}
\label{Solution formula for 1d zero initial inner source wave equ} 
W(t,s)=\frac{1}{2}\int_0^t \int_{s-\tau}^{s+\tau}F(t-\tau,y) \;dyd\tau, \: s\in \R, \: t\in [0,2T],
\end{equation}
solves \eqref{1D wave equ} (see e.q. \cite{Evans}). 
By the change of variables $T-s = \tau$, we conclude
$$
W(T,T)=\frac{1}{2}\int_0^T \int_{\tau}^{2T-\tau}F(\tau,y) \;dyd\tau.
$$
$$
=\langle f,J\Lambda_{\Src} h\rangle_{L^2(\Src \times (0,T))}-\langle\Lambda_{\Src} f,J h\rangle_{L^2(\Src\times (0,T))}.
$$
\end{proof}
\begin{lemma}
\label{Adjoint of solution mapping}
The adjoint mapping of $\Lambda_{\Src}$ in $L^2((0,T)\times \Src)$ is $ R \Lambda_{\Src} R$,
where 
$$
R h(t,x)=h(T-t,x).
$$
\end{lemma}
\begin{proof}
Let $f,h\in \mathcal{F}_{\Src,T}$ and consider the wave equations
\begin{equation}
\label{eq_wave_pair}
\left\lbrace \begin{array}{l}
(\p^2_t-\Delta_g)w=f, \: \textrm{ in } (0,T) \times N
\\
w|_{t=0}=\p_tw|_{t=0}=0
\end{array} \right. \textrm{ and } \left\lbrace \begin{array}{l}
(\p^2_t-\Delta_g)u=h, \: \textrm{ in } (0,T)\times N
\\
u|_{t=T}=\p_tu|_{t=T}=0.
\end{array} \right.
\end{equation}
We start with observing that
$$
\langle f,u\rangle _{L^2((0,T)\times N)}-\langle w, h\rangle_{L^2((0,T)\times N)}=0.
$$
This holds due the computations we have done in the proof of Theorem \ref{density}.
Therefore
$$
\langle f,u \rangle _{L^2((0,T)\times \Src)}-\langle \Lambda_{\Src}f, h\rangle_{L^2((0,T)\times \Src))}=0 \hbox{ and } \Lambda^{\ast}_{\Src}h=u|_{(0,T)\times \Src} .
$$
Replace $f=Rh$. Then 
$$
\square Ru=\square u(T-\cdot,\cdot)=h(T-\cdot,\cdot)=Rh \textrm{ and } Ru(0,\cdot)=\partial_tRu(0,\cdot)=0.
$$
By \eqref{eq_wave_pair} it holds that 
$$Ru|_{(0,T)\times \Src}=w|_{(0,T)\times \Src}=\Lambda_{\Src}f=\Lambda_{\Src}Rh.$$
Since $R\circ R=id_{L^2((0,T)\times \Src)}$ we get $u|_{(0,T)\times \Src}=R\Lambda_{\Src}Rh.$
\end{proof}
\nocite{uhlmann1998inverse,krupchyk2008inverse,de2014reconstruction,lassas2014inverse}

\bibliographystyle{abbrv}

\bibliography{bibliography}

\ifdraft{
\listoftodos
}{}
\end{document}